\documentclass[amssymb,11pt]{amsart}
\usepackage{amscd,amssymb,euscript,amsthm}
\theoremstyle{plain}
\newtheorem{thm}{Theorem}[section] 
\newtheorem{cor}[thm]{Corollary}
\newtheorem{prop}[thm]{Proposition}
\newtheorem{lem}[thm]{Lemma}

\theoremstyle{definition}
\newtheorem{defn}[thm]{Definition}
\theoremstyle{remark}
\newtheorem{rem}[thm]{Remark}

\numberwithin{equation}{section}
\newcommand{\rank}{\operatorname{rank}}

\newcommand{\tr}{\operatorname{trace}}

\newcommand{\Sep}{\operatorname{Sep}}
\def\<{\left<}
\def\>{\right>}
\def\cstar{$C^*$-algebra}
\begin{document}
\title{The probability of entanglement}
\author{William Arveson}
%
%
%
%
\subjclass[2000]{Primary 46N50; Secondary 81P68, 94B27}
\date{26 December, 2007}

\begin{abstract} 
We show that states on tensor products of matrix algebras 
whose ranks are relatively small are {\em almost surely} entangled, but 
that states of maximum rank are not.    
More precisely, let 
$M=M_m(\mathbb C)$ and $N=M_n(\mathbb C)$ be full matrix algebras  with $m\geq n$, 
fix an arbitrary state $\omega$ of $N$, and let $E(\omega)$ be the set of all states of 
$M\otimes N$ that extend $\omega$.  The space $E(\omega)$ contains states of rank $r$ for every 
$r=1,2,\dots,m\cdot\rank\omega$, and it has a filtration into compact subspaces 
$$
E^1(\omega)\subseteq E^2(\omega)\subseteq \cdots\subseteq E^{m\cdot\rank\omega}=E(\omega), 
$$  
where $E^r(\omega)$ is the set of all states of $E(\omega)$ having rank $\leq r$.  

We show first that for every $r$, there is a real-analytic manifold $V^r$,  
homogeneous 
under a transitive action of a compact group $G^r$,  
which parameterizes $E^r(\omega)$.  The unique $G^r$-invariant probability measure 
on $V^r$ promotes to a probability measure $P^{r,\omega}$ on $E^r(\omega)$, 
and $P^{r,\omega}$ assigns 
probability $1$ to states of rank $r$.  The resulting 
probability space $(E^r(\omega),P^{r,\omega})$ represents 
``choosing a rank $r$ extension of $\omega$ at random".  

Main result:  For every $r=1,2,\dots,[\rank\,\omega/2]$, 
states of $(E^r(\omega),P^{r,\omega})$ are almost surely entangled.   
\end{abstract}
\maketitle

\section{Introduction}\label{S:in}

In the literature of physics  and  quantum information theory, 
a state $\rho$ of the tensor product of two matrix algebras 
$M\otimes N$ is said to be {\em separable} (or {\em classically correlated}) if 
it is a convex combination of product states 
$$
\rho= t_1\cdot \sigma_1\otimes\tau_1+t_2 \cdot\sigma_2\otimes\tau_2+\cdots+t_r\cdot\sigma_r\otimes \tau_r, 
$$
where the coefficients $t_k$ are nonnegative and sum to $1$, and where $\sigma_k,\tau_k$ are 
states of $M$ and $N$ respectively \cite{wern1}.      
Remark \ref{enRem1} below implies that the set of separable states is a compact convex 
subset of the state space of $M\otimes N$.  A state that is not separable is said 
to be {\em entangled}.  The so-called {\em separability problem} of determining 
whether a given state of $M\otimes N$ 
is entangled is a subject of current research \cite{HorSurvey}.   
It is considered difficult, and computationally, has been shown to be NP-hard.  The 
purpose of this paper is to show that 
{\em almost surely, a state of $M\otimes N$ of relatively small rank is entangled}.   

The set $E(\omega)$ of all extensions of a fixed state $\omega$ of $N$ to a state of 
$M\otimes N$  is a compact convex subspace of the state space of $M\otimes N$, and it admits a filtration into 
compact subspaces  
$$
E^1(\omega)\subseteq E^2(\omega)\subseteq \cdots\subseteq E^{m\cdot\rank\omega}(\omega)=E(\omega)
$$
where $E^r(\omega)$ is the space of all extensions $\rho$ of $\omega$ satisfying  $\rank \rho\leq r$.  
In Sections \ref{S:ns} through \ref{S:ss} we show that for each $r$ there is a uniquely 
determined unbiased 
probability measure $P^{r,\omega}$ on $E^r(\omega)$, and that $P^{r,\omega}$ 
is concentrated on the set of states of rank $=r$.  Hence 
the probability space 
$(E^r(\omega),P^{r,\omega})$ represents ``choosing a rank $r$ 
extension of $\omega$ at random". 
The main result below is an assertion about the probability of 
entanglement in the various probability spaces $(E^r(\omega),P^{r,\omega})$, 
namely that the probability of entanglement 
is $1$ when $r$ is relatively small (see Theorem \ref{erThm1} and Remark \ref{erRem1}).  
We also point out in Theorem \ref{diThm1} that 
this behavior does not persist through large values of $r$, since 
for $r=m\cdot\rank\omega$, the probability $p$ of entanglement satisfies $0<p<1$.

\begin{rem}[Terminology and conventions]  Let $H$ be a finite dimensional Hilbert space.  
A state $\rho$ of 
$\mathcal B(H)$ has an associated density operator $A\in \mathcal B(H)$, defined by 
$\rho(X)=\tr(AX)$, $X\in \mathcal B(H)$.  
In the literature of quantum information theory, 
the operation of restricting $\rho$ to a subfactor $\mathcal N\subseteq \mathcal B(H)$ corresponds to a 
``partial tracing" operation on its density operator, in which $A\in \mathcal B(H)$ is mapped to 
the operator $\bar A\in \mathcal N$ that is defined uniquely by
\begin{equation}\label{exEq}
\rho(Y)=\tr_\mathcal N(\bar AY), \qquad Y\in \mathcal N, 
\end{equation}
where $\tr_\mathcal N$ denotes the trace of $\mathcal N$ {\em normalized so that it 
takes the value $1$ on minimal 
projections of $\mathcal N$}.  In more operator-algebraic terms, 
the partial trace of $A$ is $\bar A=\mu\cdot E(A)$, where $E:\mathcal B(H)\to \mathcal N$ is the conditional 
expectation defined by the trace of $\mathcal B(H)$ (with any normalization) and $\mu$ is the multiplicity 
of the representation of $\mathcal N$ associated with the inclusion $\mathcal N\subseteq \mathcal B(H)$.  
The constant $\mu$ is forced on the formula $\bar A=\mu\cdot E(A)$ by the normalization 
specified for  
$\tr_\mathcal N$ in (\ref{exEq}), and this non-invariant feature of (\ref{exEq}) 
leads to a problem if one attempts 
to interpret it for more general $*$-subalgebras $\mathcal N\subseteq \mathcal B(H)$.   
More significantly, the right side of 
(\ref{exEq}) loses all meaning for type $III$ subfactors $\mathcal N\subseteq \mathcal B(H)$ when $H$ is 
infinite dimensional - a situation of some importance for algebraic quantum field theory.  We 
choose to avoid such issues by dealing with restrictions and extensions of states rather 
than partial traces of operators and their inverse images.  
\end{rem} 

\begin{rem}[Literature and related results]  A significant part of the literature 
of physics and quantum information theory 
makes some connection with probabilistic aspects of entanglement.   
The following papers (and references therein) represent a sample.    
The papers \cite{szVol}, \cite{aubSz} concern Hilbert spaces 
$H_N=(\mathbb C^2)^{\otimes N}$ for large $N$, and sharp estimates 
are obtained for the smallness 
of the ratio of the volume of separable states to the volume of all states.  
In \cite{parthaMaxDim}, the 
maximal dimension of a linear subspace of $H_1\otimes\cdots\otimes H_N$ 
that contains no nonzero product vectors is calculated, 
and  in \cite{hlw} it is shown that random subspaces of 
$H\otimes K$ are likely to contain only 
near-maximally entangled vectors.  
 \cite{lockhart} discusses ``minimal" decompositions for separable states 
into convex combinations of pure product states 
(also see \cite{uhl}, \cite{sanEtAl}).  The 
survey \cite{pittRubin} also deserves mention.  
For early results on the existence of a separable ball 
in the state space see \cite{brEtAl}.   A probabilistic study of separable states 
is carried out in \cite{zyEtAl}, where lower and upper bounds are obtained 
for the probability of the set of separable states.  Those authors 
make use a rather different 
probability space, and there appears to be 
negligible overlap between \cite{zyEtAl} and this paper.  
Finally, the paper \cite{pgEtAl} concerning maximal violations of 
Bell's inequalities for tripartite systems certainly bears on 
issues of entanglement.   
\end{rem}

\begin{rem}[Convex hulls of sets in $\mathbb R^k$]\label{enRem1}
We recall some basic lore of convexity theory.  
A classical result of Carath\'eodory \cite{cara1}, \cite{cara2} asserts that 
every convex combination of points from a subset $E$  of 
$\mathbb R^k$ can be written as a convex combination of at 
most $k+1$ points of $E$.  It follows 
that the convex hull of a {\em compact} subset $E$  
of $\mathbb R^k$ is compact.  Since the set of all product states 
of $M\otimes N$ is compact, we conclude that {\em the 
set of separable states of $M\otimes N$ is compact as well as convex, and the set of entangled 
states is a relatively open subset of the state space of $M\otimes N$.}

One can do slightly better for states.  
Let  $H$ be an $n$ dimensional Hilbert space.  The self-adjoint operators 
in $\mathcal B(H)$ form a real vector space of dimension $n^2$, and 
the set of self-adjoint operators $A$ satisfying 
$\tr A=1$ is a hyperplane of dimension $n^2-1$.  
So Caratheodory's theorem implies 
that every state of $\mathcal B(H)$ 
that belongs to the convex hull of an arbitrary set $\mathcal P$ of states can be 
written as a convex combination of at most $n^2$ states of $\mathcal P$.   
\end{rem}

The proof of 
Theorem \ref{erThm1} depends on the properties of 
a numerical invariant of states of tensor products of matrix algebras - called the wedge invariant - 
that can detect entanglement.  In this section we give a precise definition of the wedge invariant, 
deferring proofs to later sections, and 
follow that with some general remarks on how the wedge invariant enters into the proof of 
Theorem \ref{erThm1}.    

Its definition requires 
that we work with operators rather than matrices, hence we 
shift attention to states $\rho$ defined on  concrete operator 
algebras $\mathcal B(K)\otimes\mathcal B(H)\cong\mathcal B(K\otimes H)$, where $H$ and $K$ are finite dimensional Hilbert 
spaces.    Fix a state $\rho$ of $\mathcal B(K\otimes H)$, let $r$ be the rank of its density operator,  
and choose vectors $\zeta_1,\dots,\zeta_r\in K\otimes H$ such that 
\begin{equation}\label{inEq0}
\rho(x)=\sum_{k=1}^r\langle x\zeta_k,\zeta_k\rangle, \qquad x\in \mathcal B(K\otimes H).  
\end{equation}
The vectors $\zeta_k$ need not be eigenvectors of the density operator of $\rho$, but 
necessarily they are linearly independent.  
Let $\omega$ be the state of $\mathcal B(H)$ defined by 
restriction
\begin{equation}\label{inEq1}
\omega(x)=\rho(\mathbf 1_K\otimes x), \qquad x\in \mathcal B(H).  
\end{equation}
The rank of $\omega$ depends on $\rho$, and can be any integer from $1$ to $n=\dim H$.  Fix a 
Hilbert space $K_0$ of dimension $\rank\omega$, such as 
$K_0=\mathbb C^{\,\rank\omega}$.  The basic GNS construction applied to $\omega$, together with the 
basic representation theory of matrix algebras,  leads to the existence of a unit 
vector $\xi\in K_0\otimes H$ that is cyclic for the algebra $\mathbf 1_{K_0}\otimes\mathcal B(H)$, 
and has the property
\begin{equation}\label{inEq2}
\omega(x)=\langle (\mathbf 1_{K_0}\otimes x)\xi,\xi\rangle, \qquad x\in\mathcal B(H).
\end{equation}
We have been asked by the referee to point out that 
this procedure of passing from $\omega$ to  the vector state defined 
by $\xi$ is known as {\em purification} in the physics literature.  

Fixing such a unit vector $\xi$, we define 
an $r$-tuple of operators $v_1,\dots,v_r$ as follows.  
Because of (\ref{inEq1}) and (\ref{inEq2}), 
one can show 
that for each $k=1,\dots,r$ there 
is a unique operator $v_k: K_0\to K$ such that 
$$
(v_k\otimes x)\xi=(\mathbf 1_K\otimes x)\zeta_k,  \qquad x\in\mathcal B(H)
$$
and one finds that $v_1, \dots, v_r\in\mathcal B(K_0,K)$ satisfies $v_1^*v_1+\cdots+v_r^*v_r=\mathbf 1_{K_0}$.  The 
$r$-tuple $(v_1,\dots,v_r)$ depends on the choice of  $\zeta_1, \dots,\zeta_r$ as well as the 
choice of $\xi\in K_0\otimes H$.  But  it is 
also a fact that if $\zeta_1^\prime, \dots, \zeta_r^\prime$ 
is another set of $r$ vectors that satisfies (\ref{inEq0}) and $\xi^\prime$ is another 
cyclic vector satisfying (\ref{inEq2}),  then the resulting $r$-tuple of operators 
$(v_1^\prime,\dots,v_r^\prime)$ is related to $(v_1,\dots,v_r)$ as follows 
\begin{equation}\label{inEq3}
v_i^\prime = \sum_{j=1}^r \lambda_{ij} v_jw, \qquad 1\leq i\leq r,   
\end{equation}
where $(\lambda_{ij})$ is a unitary $r\times r$ matrix of scalars 
and $w$ is a unitary operator in $\mathcal B(K_0)$ (see Section \ref{S:wi}).  

For every choice of integers $i_1,\dots,i_r$ with 
$1\leq i_1,\dots,i_r\leq r$ 
the tensor product of operators $v_{i_1}\otimes \cdots \otimes v_{i_r}$ belongs to $\mathcal B(K_0^{\otimes r},K^{\otimes r})$.  
Hence we can define an operator $v_1\wedge\cdots\wedge v_r\in \mathcal B(K_0^{\otimes r},K^{\otimes r})$ 
as the alternating average 
\begin{equation}\label{inEq4}
v_1\wedge\cdots\wedge v_r=\frac{1}{|G|}\sum_{\pi\in G} (-1)^\pi v_{\pi(1)}\otimes\cdots\otimes v_{\pi(r)}, 
\end{equation}
the sum extended over the group $G$ all permutations $\pi$ of $\{1,\dots,r\}$.  The permutation group 
$G$ acts naturally as unitary operators  on both $K_0^{\otimes r}$ and $K^{\otimes r}$, 
and we may form their symmetric and antisymmetric subspaces.  For example,  
in terms of the unitary representation $\pi\mapsto U_\pi$ of $G$ on $K^{\otimes r}$, 
\begin{align*}
K^{\otimes r}_+&=\{\zeta\in K^{\otimes r}: U_\pi\zeta=\zeta, \quad \pi\in G\}, \\
K^{\otimes r}_-&=\{\zeta\in K^{\otimes r}: U_\pi\zeta=(-1)^\pi\zeta, \quad \pi\in G\}.  
\end{align*}
The operator $v_1\wedge\cdots\wedge v_r$ maps the symmetric subspace 
of $K_0^{\otimes r}$ to the antisymmetric subspace of $K^{\otimes r}$, hence 
its restriction to $K^{\otimes r}_{0+}$ is an operator in $\mathcal B(K^{\otimes r}_{0+},K^{\otimes r}_-)$.  
This operator also 
depends on the choice of $\xi$,  $\eta_1,\dots,\eta_r$.  However, because of (\ref{inEq3}), 
the rank of $v_1\wedge\cdots\wedge v_r\restriction_{K^{\otimes r}_{0+}}$ is a well-defined nonnegative integer 
that we associate 
with the state $\rho$
$$
w(\rho)=\rank(v_1\wedge\cdots\wedge v_r\restriction_{K^{\otimes r}_{0+}}).  
$$

In a similar way, we may form the wedge product of the $r$-tuple of adjoints $v_k^*: K\to K_0$ 
to obtain an operator $v_1^*\wedge\cdots\wedge v_r^*\in\mathcal  B(K^{\otimes r}, K_0^{\otimes r})$, 
and restrict it to the symmetric subspace $K^{\otimes r}_+\subseteq K^{\otimes r}$ to obtain a second integer 
$w^*(\rho)=\rank(v_1^*\wedge\cdots\wedge v_r^*\restriction_{K^{\otimes r}_+})$.  
Thus we can 
make the following 
\begin{defn}\label{wsDef1}
The {\em wedge invariant} of a state $\rho$ of $\mathcal B(K\otimes H)$ is defined as the pair 
of nonnegative integers $(w(\rho), w^*(\rho))$, where 
$$
w(\rho)=\rank(v_1\wedge\cdots\wedge v_r\restriction_{K^{\otimes r}_{0+}}), 
\quad w^*(\rho)=\rank(v_1^*\wedge\cdots\wedge v_r^*\restriction_{K^{\otimes r}_+}).  
$$
\end{defn}

The wedge invariant has two principal features.  First, it is 
capable of detecting entanglement because of the following result of Section \ref{S:wi}:   

\begin{thm}\label{wsThm1}
If $\rho$ is a separable state of $\mathcal B(K\otimes H)$, 
then $w(\rho)\leq 1$ and $w^*(\rho)\leq 1$.  
\end{thm}

This separability criterion differs fundamentally from others that involve 
positive linear maps (see \cite{peresSep} and \cite{storSep}).  

The second feature of the wedge invariant is that 
it is associated with subvarieties of the real algebraic varieties that 
will be used to parameterize states in the following sections.  
To illustrate that geometric feature in broad terms, let $Y$ and $Z$ be finite-dimensional 
complex vector spaces, let $\mathcal B(Y,Z)$ be the space of all linear operators 
from $Y$ to $Z$, and consider the set  $\mathcal B(Y,Z)^r$ of all $r$-tuples 
$v=(v_1,\dots,v_r)$ 
with components $v_k\in\mathcal B(Y,Z)$.  Then for every $k=1,2,\dots$, the set 
of $r$-tuples 
$$
W^r(k)=\{v=(v_1,\dots,v_r)\in \mathcal B(Y,Z)^r: \rank(v_1\wedge\cdots\wedge v_r\restriction_{Y_+^{\otimes r}})\leq k\} 
$$
is an algebraic set - namely the set of common zeros of a finite set $f_1,\dots, f_p$ of 
real-homogeneous multivariate polynomials $f_k:\mathcal B(Y,Z)^r\to \mathbb R$.  
This leads to the following fact that provides 
a key step in the proof of Theorem \ref{erThm1} below: {\em Let $r=1,2,\dots$ and let 
$M$ be a $d$-dimensional connected real-analytic submanifold of $\mathcal B(Y,Z)^r$ 
that contains a point $(v_1,\dots,v_r)\in M$ for which 
\begin{equation}\label{wsEq2}
\rank(v_1\wedge\cdots\wedge v_r\restriction_{Y_+^{\otimes r}})>k
\end{equation}
for some $k\geq 1$.  Then (\ref{wsEq2}) is generic in the sense that for every relatively 
open subset $U\subseteq M$ 
endowed with real-analytic coordinates, $U\cap W^r(k)$ is a set of $d$-dimensional Lebesgue measure zero.  
}
\vskip0.1in

The methods we use are a mix of matrix/operator theory, 
convexity, and basic real algebraic geometry.  
In Section \ref{S:ec}, we offer some general remarks that address the broader issue of whether 
one can expect an effective ``real-analytic" characterization of entanglement in general.  
Finally, for the reader's convenience we have included two appendices containing formulations 
of some known results about real-analytic varieties of 
matrices that are fundamental for the analysis of Sections \ref{S:ns} through \ref{S:di}.  

We also point out that further applications to completely positive maps on matrix algebras 
are developed in a sequel to this paper \cite{arvEnt2}.

\section{The noncommutative spheres $V^r(n,m)$}\label{S:ns}

Let $m, n$ be positive integers with $m\geq n$.  For every $r=1,2,\dots$, we work with 
the space $V^r(n,m)$ of all $r$-tuples $v=(v_1,\dots,v_r)$ of complex 
$m\times n$ matrices $v_k$ such that 
\begin{equation}\label{nsEq1}
v_1^*v_1+\cdots+v_r^*v_r=\mathbf 1_n.
\end{equation}
There is a natural left action of the unitary group $U(rm)$ on $V^r(n,m)$, defined as follows.  
An element of $U(rm)$ can be viewed as a unitary 
$r\times r$ matrix $w=(w_{ij})$ with entries $w_{ij}$ in the matrix algebra $M_m(\mathbb C)$, 
and it acts on an element $v=(v_1,\dots,v_r)\in V^r(n,m)$ by way of $w\cdot v=v^\prime$, where 
\begin{equation}\label{nsEq1.2}
v_i^\prime=\sum_{j=1}^r w_{ij}v_j,\qquad 1\leq i\leq r.  
\end{equation}

There is also a right action of the 
unitary group $U(n)$ on $V^r(n,m)$, in which $u\in U(n)$ acts on $v\in V^r(n,m)$ by 
$(v_1,\dots,v_r)\cdot u=(v_1u,\dots,v_ru)$.  Both actions are 
better understood in terms of operators, 
after the identifications of the following paragraph have been made.  

\subsection{The varieties $V^r(H,K)$}

Note that $n$ precedes $m$ in the notation for $V^r(n,m)$.  This convention arises from the interpretation 
of $V^r(n,m)$ as a space of operators rather than matrices.  
If $H$ and $K$ are complex Hilbert spaces of respective dimensions $n$ and $m$, then the space 
$V^r(H,K)$ of all $r$-tuples of operators $v=(v_1,\dots,v_r)$ with components $v_k\in\mathcal B(H,K)$ 
that satisfy the counterpart of (\ref{nsEq1}), 
\begin{equation}\label{nsEq1.3}
v_1^*v_1+\cdots+v_r^*v_r=\mathbf 1_H,
\end{equation}
can 
be identified with $V^r(n,m)$ after making a choice of orthonormal bases for both $H$ and $K$, and all 
statements about $V^r(n,m)$ have appropriate counterparts in the more coordinate-free context 
of the spaces $V^r(H,K)$.  Throughout  this paper, it will serve our purposes better to 
interpret $V^r(n,m)$ as the space of $r$-tuples of operators $V^r(H,K)$.  

$V^r(H,K)$ is a compact subspace of the complex vector space $\mathcal B(H,K)^r$ of all $r$-tuples 
of operators $v=(v_1,\dots,v_r)$ with components in $\mathcal B(H,K)$, on which the unitary group $U(r\cdot K)$ 
of the direct sum $r\cdot K$ of $r$ copies of $K$ acts 
smoothly on the left.  Because of the presence of the $*$-operation in (\ref{nsEq1.3}), we 
can also view 
the ambient space $\mathcal B(H,K)^r$ as a finite dimensional {\em real} vector space, endowed    
with the (real) inner product 
\begin{equation}\label{nsEq2}
\langle (v_1,\dots,v_r),(w_1,\dots,w_r)\rangle =\Re \sum_{k=1}^r \tr w_k^*v_k, \qquad v,w\in \mathcal B(H,K)^r.  
\end{equation}
The following result summarizes the geometric structure that $V^r(H,K)$ inherits from 
its ambient space, when $H$ and $K$ are Hilbert spaces satisfying $n=\dim H\leq m=\dim K<\infty$.  

\begin{thm}\label{nsThm1}
For every $r=1,2,\dots$, the space $V^r(H,K)$ is a compact, connected, 
real-analytic Riemannian manifold of dimension $d=n(2rm-n)$, on which the unitary group
$\mathcal U(r\cdot K)$ acts as a transitive group of isometries.  In particular, the natural 
measure associated with its Riemannian metric is proportional to the unique 
probability measure on $V^r(H,K)$ that is invariant under the transitive $\mathcal U(r\cdot K)$-action.  
\end{thm}

\begin{proof}
We 
identify the space $\mathcal B(H,K)^r$ of $r$-tuples of operators 
as the space  $\mathcal B(H,r\cdot K)$ of all 
operators from $H$ into the direct sum $r\cdot K$ of $r$ copies of $K$, in which an 
$r$-tuple $v=(v_1,\dots,v_r)$ of operators in $\mathcal B(H,K)$ 
is identified with the single operator $\tilde v:H\to r\cdot K$ defined by 
$$
\tilde v\xi=(v_1\xi,\dots,v_r\xi), \qquad \xi\in H.  
$$
After this identification, $V^r(H,K)$ becomes the space of all isometries 
in $\mathcal B(H,r\cdot K)$, and Theorem \ref{a1Thm2} implies that $V^r(H,K)$ inherits the 
structure of a connected real-analytic submanifold of the ambient real vector space 
$\mathcal B(H,r\cdot K)\cong\mathcal B(H,K)^r$ in which it is embedded, and that the 
unitary group $\mathcal U(r\cdot K)$ acts transitively on it by left multiplication.

The 
inner product (\ref{nsEq2}) on $\mathcal B(H,K)^r$ restricts so as to give 
 a Riemannian metric on the tangent 
bundle of $V^r(H,K)$, thereby making it into a compact Riemannian manifold.  

Notice that the action 
of $\mathcal U(r\cdot K)$ is actually defined on the larger inner product 
space $\mathcal B(H,K)^r$, and 
its action on $\mathcal B(H,K)^r$  is by isometries.  Indeed, let 
$u\in U(r\cdot K)$,  and view $u$ as an $r\times r$ matrix $(u_{ij})$ of operators $u_{ij}$ 
in $\mathcal B(K)$.  
Choosing $v,w\in \mathcal B(H,K)^r$ and setting $v^\prime=u\cdot v$ and 
$w^\prime=u\cdot w$ as in (\ref{nsEq1.2}), then $\sum_k u_{ki}^*u_{kj}=\delta_{ij}\mathbf 1_K$ 
because $u=(u_{ij})$ is unitary, hence 
\begin{align*}
\langle v^\prime,w^\prime\rangle &= \Re\sum_{k=1}^r\tr (w_k^{\prime *}v_k^\prime)=
\Re\sum_{i,j,k=1}^r\tr(w_i^*u_{ki}^*u_{kj}v_j)
\\
&=\Re \sum_{i=1}^r\tr(w_i^*v_i)=\langle v,w\rangle.  
\end{align*}
Hence $\mathcal U(r\cdot K)$ acts as 
isometries on the Riemannian submanifold $V^r(H,K)$.  

Finally, the dimension calculation amounts to little more than 
subtracting the number of real equations appearing 
in the matrix equation (\ref{nsEq1}) from the real dimension $\dim_\mathbb R(\mathcal B(H,K)^r)$ 
of the vector space  $\mathcal B(H,K)^r$.  
\end{proof}

\begin{rem}\label{upsRem1.5}[Right action of $\mathcal U(H)$ on $V^r(H,K)$]
The right action of the unitary group $\mathcal U(H)$ on $r$-tuples of operators in $\mathcal B(H,K)^r$ is defined by 
$$
(v,w)\in \mathcal B(H,K)^r\times \mathcal U(H)\mapsto v\cdot w=(v_1w,\dots,v_rw).  
$$ 
This action of $\mathcal U(H)$ commutes with the left action of $\mathcal U(r\cdot K)$ and 
it preserves the inner product of $\mathcal B(H,K)^r$.  Hence it restricts to a right 
action of $\mathcal U(H)$ on $V^r(H,K)$ that commutes with the transitive left action, and  
which also acts as isometries relative to the Riemannian structure of $V^r(H,K)$.  
\end{rem}

\begin{rem}\label{upsRem2}[The invariant measure class of $V^r(H,K)$]
Perhaps it is unnecessary to point out that the natural measure class of $V^r(H,K)$  is 
that of Lebesgue measure in local coordinates; more precisely, {\em relative to real-analytic local 
coordinates on an open subset  of $V^r(H,K)$, the measure 
$\mu$ associated with the Riemannian metric is mutually absolutely continuous with 
the transplant of Lebesgue measure to that chart.  }
\end{rem}

\subsection{Subvarieties of $V^r(H,K)$}

There is an intrinsic notion of real-analytic function $f: V^r(H,K)\to \mathbb R$, namely 
a function such that for every real-analytic isomorphism  $u:D\to U$ of an open ball $D\subseteq \mathbb R^d$ 
onto an open set $U\subseteq V^r(H,K)$, $f\circ u$ is a real-analytic function on $D$ 
(see Appendix \ref{S:a1}).  
Similarly, given a finite dimensional real vector space $W$, one can speak of 
real-analytic functions 
\begin{equation}\label{nsEq5}
F: V^r(H,K)\to W, 
\end{equation}
and though it is rarely necessary to do so, one 
can reduce the analysis of such vector functions 
to that of $k$-tuples of real-valued analytic functions by composing $F$ with a basis of 
linear functionals $\rho_1,\dots,\rho_k$ for the dual of $W$.

\begin{rem}[Homogeneous polynomials]\label{nsRem1}
Virtually all of the analytic functions (\ref{nsEq5}) that we will encounter 
are obtained by restricting homogeneous polynomials defined on 
the ambient space $\mathcal B(H,K)^r$ to $V^r(H,K)$.  
Let $V$ and $W$ be 
finite dimensional real vector spaces.  
A map $F: V\to W$ is  said to be a real homogeneous 
polynomial (of degree $k$) if it has the form $F(v)=G(v,v,\dots,v)$ where $G$ is a real multilinear 
mapping $G: V^k\to W$ in $k$ variables.  Though this terminology is 
slightly abusive in that the zero function qualifies as a homogeneous 
polynomial of every positive degree, it will not cause problems in this paper.      
A function $F:V\to W$ is a homogeneous polynomial of degree $k$ iff $\rho\circ F$ is a scalar-valued homogeneous polynomial 
of degree $k$ for every linear functional $\rho: W\to \mathbb R$.  
\end{rem}

\begin{defn}\label{nsDef1}
By a {\em subvariety} of $V^r(H,K)$ we mean a subspace $Z$ of $V^r(H,K)$ of the form 
$$
Z=\{v\in V^r(H,K): F(v)=0\}, 
$$
where $F: V^r(H,K)\to W$ is a real-analytic function taking values in some finite-dimensional 
real vector space $W$.  
\end{defn} 
Subvarieties are obviously compact.  As a concrete example, the set  
$$
Z=\{v=(v_1,\dots,v_r)\in V^r(H,K): \rank v_1\leq 2\}
$$
is the zero subvariety associated with the restriction to $V^r(H,K)$ of the cubic homogeneous 
polynomial $F: \mathcal B(H,K)^r\to \mathcal B(\wedge^3 H,\wedge^3 K)$, where 
$$
F(v)=(v_1\otimes v_1\otimes v_1)\restriction_{H\wedge H\wedge H}.  
$$

\begin{prop}\label{nsProp1}
Let $Z$ be a subvariety of $V^r(H,K)$ and let $\mu$ be the natural measure 
of $V^r(H,K)$.  If $Z\neq V^r(H,K)$, then $\mu(Z)=0$.  
\end{prop}

\begin{proof}  Let $F: V^r(H,K)\to W$ be a real-analytic function taking values in 
a finite dimensional real vector space such that 
$$
Z=\{v\in V^r(H,K): F(v)=0\}.  
$$  
$F$ cannot vanish identically because $Z\neq V^r(H,K)$; and since $V^r(H,K)$ is 
connected and $F$ is real-analytic, it cannot vanish identically 
on any nonempty open subset of $V^r(H,K)$.  

Let $d=\dim(V^r(H,K))$ and let 
$\mu$ be the natural measure of $V^r(H,K)$ associated with its Riemannian metric.  To show 
that $\mu(Z)=0$, it suffices to show that every point of $V^r(H,K)$ has a neighborhood 
$U$ such that $\mu(U\cap Z)=0$.  To prove that, fix a point $v\in V^r(H,K)$ and choose 
an open neighborhood $U$ of $v$ that can be coordinatized by the open unit ball $B\subseteq \mathbb R^d$ by 
way of a real-analytic isomorphism $u: B\to U$ (see Appendix \ref{S:a1}).  The composition $F\circ u: B\to W$ is 
a real-analytic mapping that does not 
vanish identically on $B$, hence there is a real-linear functional $\rho: W\to \mathbb R$ 
such that $\rho\circ F\circ u$ 
does not vanish identically on $B$.  Since $\rho\circ F\circ u$ is a real-valued analytic function
of its variables, 
Proposition \ref{a2Prop1} implies that the set $\tilde Z$ of its zeros has Lebesgue 
measure zero.  It follows that $u(\tilde Z)\subseteq U$ is a set of $\mu$-measure zero that contains $U\cap Z$, hence 
$\mu(U\cap Z)=0$.  
\end{proof}

\section{The unbiased probability spaces $(X^r,P^r)$}\label{S:ups}

Let $H$, $K$ be Hilbert spaces, with 
$n=\dim H\leq m=\dim K<\infty$.  
In section \ref{S:ss}, we will show that the spaces $V^r(H,K)$ can be used to parameterize states 
of $\mathcal B(K\otimes H)$.  The parameterizing map is not injective, but it promotes naturally 
to an injective map of a quotient $X^r$ of $V^r(H,K)$.  We now introduce these spaces $X^r$ 
and we show that each of them carries a unique unbiased probability measure $P^r$, 
so that $(X^r,P^r)$ becomes a topological probability 
space that serves to parameterize states faithfully.    In this section we summarize the basic properties of these 
probability spaces and discuss 
some of the random variables that will enter into the analysis of states later on.  

The group $U(r)$ of all scalar $r\times r$ unitary matrices in $M_r(\mathbb C)$ is 
identified with 
a subgroup of $\mathcal U(r\cdot K)$ - consisting unitary operator matrices with 
components in $\mathbb C\cdot \mathbf 1_K$, hence it 
acts naturally on $V^r(H,K)$,  
in which $\lambda=(\lambda_{ij})\in U(r)$ acts on $v=(v_1,\dots,v_r)\in V^r(H,K)$ 
by way of $\lambda\cdot v=v^\prime$ where 
\begin{equation}\label{upsEq1}
v_i^\prime=\sum_{j=1}^r\lambda_{ij} v_j, \qquad i=1,2, \dots, r.  
\end{equation}
Since $U(r)$ is compact and acts smoothly on $V^r(H,K)$, its orbit space is a 
compact metrizable space $X^r$.  Moreover, the natural projection 
$$
v\in V^r(H,K)\mapsto \dot v\in X^r
$$
is a continuous surjection with the following universal property that we will use 
repeatedly: {\em For every topological 
space $Y$ and every 
continuous function $f: V^r(H,K)\to Y$ satisfying $f(\lambda\cdot v)=f(v)$ for $\lambda\in U(r)$, 
$v\in V^r(H,K)$, there is a unique continuous function $\dot f: X^r\to Y$ 
such that $\dot f(\dot v)=f(v)$, $v\in V^r(H,K)$.  }  Note too that the commutative 
\cstar\ $C(X^r)$ is isomorphic to the $C^*$-subalgebra $A\subseteq C(V^r(H,K))$ of functions 
$f\in C(V^r(H,K))$ that satisfy $f(\lambda\cdot v)=f(v)$ for $\lambda\in U(r)$, 
$v\in V^r(H,K)$.

It follows that the quotient space 
$X^r$ carries a unique unbiased probability measure $P^r$ that is defined on 
Borel subsets $E$ by promoting the unique invariant probability measure $\mu$ of $V^r(H,K)$
$$
P^r(E)=\mu\{v\in V^r(H,K);\ \dot v\in E\}, \qquad E\subseteq X^r.  
$$ 
Equivalently, in terms of the identification $C(X^r)\cong A\subseteq C(V^r(H,K))$ 
of the previous paragraph, $P^r$ is the measure 
on the Gelfand spectrum $X^r$ of $A$ that the Riesz-Markov theorem associates with the state  
$$
\rho(f)=\int_{V^r(H,K)}f(v)\,d\mu(v), \qquad f\in A.  
$$
In this way we obtain a compact metrizable probability space $(X^r,P^r)$.  Notice that 
$(X^r,P^r)$ depends not only on $r$, but also $H$ and $K$ - or at least on their 
dimensions $n$ and $m$.  However, since $H$ and $K$ will be fixed throughout the 
discussions to follow, we can safely lighten notation by omitting reference 
to these extra parameters.  

\begin{rem}[Right action of $\mathcal U(H)$ on $X^r$]\label{upsRem1}
Note that while the left action of the 
larger group $\mathcal U(r\cdot K)$ acts transitively on $V^r(H,K)$, that symmetry 
is lost when one passes to the orbit space  $X^r$ because $U(r)$ is not a normal subgroup 
of $\mathcal U(r\cdot K)$.  On the other hand, the right action 
of $\mathcal U(H)$ on $V^r(H,K)$ does promote naturally to a right action on $X^r$.  Moreover, 
since the right action on $V^r(H,K)$ preserves the Riemannian metric, it also preserves 
the natural measure $\mu$ of $V^r(H,K)$.  
We conclude: {\em The right action of the unitary group $\mathcal U(H)$ on $X^r$ 
gives rise to a compact group of measure-preserving homeomorphisms of the topological 
probability space $(X^r,P^r)$.}
\end{rem}

\begin{rem}[The rank variable]
We begin by defining a random variable   
$$
\rank: X^r\to \{1,2,\dots,r\}.  
$$ 
For $v=(v_1,\dots,v_r)\in V^r(H,K)$, 
let $\mathcal S_v={\rm{span\,}}\{v_1,\dots,v_r\}$ be the complex linear subspace of $\mathcal B(H,K)$ 
spanned by its component operators.  Elementary 
linear algebra shows that $\mathcal S_{\lambda\cdot v}=\mathcal S_v$
for every $\lambda=(\lambda_{ij})\in U(r)$, and in particular  
the dimension of $\mathcal S_v$ depends only on the image 
$\dot v$ of $v$ in $X^r$.  Hence we can define a function $\rank: X^r\to \{1,2,\dots,r\}$ by 
\begin{equation}\label{upsEq2}
\rank(\dot v)=\dim\mathcal S_v, \qquad v\in V^r(H,K).  
\end{equation}
Since the function $v\mapsto \dim \mathcal S_v$ is lower semicontinuous in the sense that 
$\{v\in V^r(H,K): \dim \mathcal S_v\leq k\}$ is closed for every $k$, it follows that 
the rank function is Borel-measurable, and hence defines a random variable.  Moreover, since 
$\dim \mathcal S_{v\cdot w}=\dim \mathcal S_v$ for every $w\in \mathcal U(H)$, {\em the 
rank variable is invariant under the right action of $\mathcal U(H)$ on $X^r$.}   
\end{rem}

Significantly, rank is almost surely constant throughout $X^r$:   

\begin{thm}\label{upsThm1}
For every $r=1,2,\dots,mn$, $P^r\{x\in X^r: \rank(x)\neq r\}=0$.  
\end{thm}

The proof of Theorem \ref{upsThm1} requires:  

\begin{lem}\label{upsLem0}  
For every $r=1,2,\dots,mn$, $V^r(H,K)$ contains an $r$-tuple $v=(v_1,\dots,v_r)$ with 
linearly independent component operators $v_1,\dots,v_r$.  
\end{lem}

\begin{proof}
Fixing $r$, $1\leq r\leq mn$, we claim first that there is a linearly 
independent set of operators $a_1,\dots,a_r:H\to K$ such that 
\begin{equation}\label{upsEq3}
\ker a_1\cap\cdots\cap\ker a_r=\{0\}.
\end{equation}
Indeed, since $\dim\mathcal B(H,K)=mn\geq r$, we can find a linearly independent 
subset $b_1,\dots,b_r\in\mathcal B(H,K)$.  Set $H_0=\ker b_1\cap\cdots\cap \ker b_r$ 
and let $r\cdot K$ be the direct sum of $r$ copies of $K$.  
The linear operator $B: \xi\in H\mapsto (b_1\xi,\dots,b_r\xi)\in r\cdot K$ has 
kernel $H_0$, hence $\dim BH+\dim H_0=n\leq m=\dim K\leq \dim(r\cdot K)$, and 
therefore $\dim H_0\leq \dim(r\cdot K)-\dim BH$.  
Hence there is a partial isometry $B^\prime$ in $\mathcal B(H,r\cdot K)$ with initial space 
$H_0$ and final space contained in $BH^\perp$.  Writing $B^\prime\xi=(b_1^\prime\xi,\dots,b_r^\prime\xi)$ 
with $b_k^\prime\in \mathcal B(H,K)$, we set 
$$
a_1=b_1+b_1^\prime,\ a_2=b_2+b_2^\prime, \dots ,\, a_r=b_r+b_r^\prime.  
$$
These operators restrict to a linearly independent set of operators 
from $H_0^\perp$ into $K$, hence they are linearly independent subset of $\mathcal B(H,K)$;  
and since the operator $B+B^\prime\in\mathcal B(H,r\cdot K)$ 
has trivial kernel, (\ref{upsEq3}) follows.

Fix such an $r$-tuple $a_1,\dots,a_r$.  Then  $a_1^*a_1+\cdots+a_r^*a_r$ is an invertible operator 
in $\mathcal B(H)$, and we can define a new $r$-tuple  $v_1,\dots,v_r$ in $\mathcal B(H,K)$ by 
$$
v_k=a_k(a_1^*a_1+\cdots+a_r^*a_r)^{-1/2}, \qquad k=1,\dots,r.  
$$
The operators $v_k$ are also linearly independent, and by its construction, the $r$-tuple 
$v=(v_1,\dots,v_r)$ belongs 
to $V^r(H,K)$.
\end{proof}

\begin{proof}[Proof of Theorem \ref{upsThm1}]
Consider the function $F: V^r(H,K)\to \wedge^r\mathcal B(H,K)$ 
obtained by restricting the homogeneous polynomial defined on $\mathcal B(H,K)^r$ 
$$
F(v)=v_1\wedge\cdots\wedge v_r,\qquad v=(v_1,\dots,v_r)\in \mathcal B(H,K)^r, 
$$ 
to the submanifold $V^r(H,K)$.   
Obviously, $F$ is real-analytic, and elementary 
multilinear algebra implies that for every $v=(v_1,\dots,v_r)\in V^r(H,K)$, 
$$
\{v_1,\dots,v_r\}\, {\rm{is\ linearly\ dependent}}\iff v_1\wedge\cdots\wedge v_r=0.   
$$ 
Hence $\dim \mathcal S_v<r\iff F(v)=0$.   
It follows from Lemma 
\ref{upsLem0} that the polynomial $F$ does not vanish identically on $V^r(H,K)$, 
so by Proposition \ref{nsProp1}, its zero variety 
$
Z=\{v\in V^r(H,K): F(v)=0\}
$
is a closed subset of $V^r(H,K)$ of 
$\mu$-measure zero.  Moreover, $Z$ is invariant under the left action of $U(r)$ on 
$V^r(K,K)$ because for $\lambda\in U(r)$, $v=(v_1,\dots,v_r)\in V^r(H,K)$ and 
$\lambda\cdot v=(v_1^\prime,\dots,v_r^\prime)$ as in (\ref{upsEq1}), we have 
$$
F(\lambda\cdot v)=v_1^\prime\wedge\cdots\wedge v_r^\prime=\det(\lambda_{ij})\cdot v_1\wedge\cdots\wedge v_r=
\det(\lambda_{ij})\cdot F(v).  
$$
It follows that $\dot Z$ is a closed set of probability zero in $X^r$, 
$$
P^r(\{x\in X^r: \rank(x)<r\})=P^r(\dot Z)=\mu(Z)=0, 
$$
and Theorem \ref{upsThm1} follows.    
\end{proof}

\section{Operators associated with extensions of states}\label{S:oe}  

Let $H_0$ be a finite dimensional Hilbert space and let $N\subseteq \mathcal B(H_0)$ be 
a subfactor - a $*$-subalgebra with trivial center that contains the identity operator.  
Every state $\omega$ of $N$ can be extended in many ways 
to a state of $\mathcal B(H_0)$.  In this section we show that the range of 
the density operator of every extension $\rho$ is linearly isomorphic to 
a certain operator space associated with the pair $(\rho,\omega)$.  
While this identification is technically straightforward, it seems not to be 
part of the lore of matrix algebras.  The details follow.    

For every state $\omega$ of $N$, 
the set $E(\omega)$ of all extensions of $\omega$ to a state of $\mathcal B(H_0)$
is a compact convex subset of the state space of $\mathcal B(H_0)$.  We begin 
with some elementary observations that relate 
properties of $\omega$ to properties of the various states in $E(\omega)$.  The 
{\em support projection} of a state $\rho$ of $\mathcal B(H_0)$ is 
defined as the smallest projection 
$p\in\mathcal B(H_0)$ such that $\rho(p)=1$; the range $pH_0$ of the support projection 
of $\rho$ is the same as the range of its density operator, and the dimension of 
that space is called the {\em rank} of $\rho$.

\begin{lem}\label{oeLem0}  Let $N\subseteq \mathcal B(H_0)$ be a subfactor, let 
$\omega$ be a state of $N$, and let $p$ be the smallest projection in $N$ satisfying $\omega(p)=1$.  
Then the range of the density operator of every 
state in $E(\omega)$ is contained in $pH_0$.  
\end{lem}

\begin{proof}
Choose $\rho\in E(\omega)$.  Since $p\in N$, we have $\rho(p)=\omega(p)=1$.  It follows 
that the support projection $q\in\mathcal B(H_0)$ of $\rho$ satisfies 
$q\leq p$.   
\end{proof}

\begin{rem}[Extensions of faithful states]\label{ssRem1}
It is significant that for purposes of analyzing the structure of $E(\omega)$, {\em 
one can restrict attention to extensions of faithful states  $\omega$.  }
Indeed, letting $p$ be as in Lemma \ref{oeLem0},  we see that since 
every state in $E(\omega)$ is supported in $pH_0$, it can be viewed as a 
state of $\mathcal B(pH_0)=p\mathcal B(H_0)p$ 
that extends the faithful state defined by restricting 
$\omega$ to the corner $pNp\subseteq N$.  Since $pNp$ is 
also a subfactor of $\mathcal B(pH_0)$, the asserted reduction is apparent.    
\end{rem}

\begin{rem}[Commutants and tensor products]\label{ssRem2}
Let $M=N^\prime$ be the commutant of $N$ in $\mathcal B(H_0)$.  $M$ is also a subfactor, and 
we can identify the \cstar\  $\mathcal B(H_0)$ with $M\otimes N$.  
Since we intend to discuss entanglement 
among the states of $E(\omega)$, it is better to view $E(\omega)$ 
as the set of states $\rho$ on the tensor product $M\otimes N$ that satisfy 
$$
\rho(b)=\omega(\mathbf 1_M\otimes b),\qquad b\in N.  
$$
Having made these identifications, we are free to 
introduce new ``coordinates" that realize $M$ as $\mathcal B(K)$, 
$N$ as $\mathcal B(H)$, and $M\otimes N$ as $\mathcal  B(K\otimes H)$.  
\end{rem}

\begin{rem}[Mixed states of $N$]  Since every extension of a pure state $\omega$ 
of $N$ to $M\otimes N$ is easily seen to be separable, the separability 
problem has content only for extensions to $M\otimes N$ of {\em mixed} states $\omega$.  
In view of Remark \ref{ssRem1}, {\em we should analyze 
extensions of faithful states of $N$ to $M\otimes N$ in cases 
where $N=\mathcal B(H)$ and 
$\dim H\geq 2$.}  

\end{rem}

We collect the following elementary fact -- a textbook exercise 
on the GNS construction and the representation theory of matrix algebras.  

\begin{lem}\label{oeLem1}
Let $H$ be a finite-dimensional Hilbert space and let $\omega$ be a state of $\mathcal B(H)$ 
of rank $r$.  Then there is a unit vector $\xi_\omega\in \mathbb C^r\otimes H$ such that 
$\omega(b)=\langle(\mathbf 1_{\mathbb C^r}\otimes b)\xi_\omega, \xi_\omega\rangle$, $b\in\mathcal B(H)$, and 
$\xi_\omega$ is a cyclic vector for the algebra $\mathbf 1_{\mathbb C^r}\otimes\mathcal B(H)$.  If $\xi_\omega^\prime$ 
is another vector in $\mathbb C^r\otimes H$ with the same property, then there is a 
unique unitary operator $w\in \mathcal B(\mathbb C^r)$ 
such that $\xi_\omega^\prime=(w\otimes\mathbf 1_H)\xi_\omega$.  
\end{lem}

\begin{prop}\label{oeProp1}
Let $\omega$ be a state of $\mathcal B(H)$, let $K_0$ be a Hilbert space of 
dimension $\rank\omega$, and let 
$$
\omega(b)=\langle (\mathbf 1_{K_0}\otimes b)\xi_\omega,\xi_\omega\rangle, \qquad b\in\mathcal B(H)
$$ 
be a representation of $\omega$ with the properties of Lemma \ref{oeLem1}.

For every state $\rho$ of $\mathcal B(K\otimes H)$ that restricts to $\omega$ 
$$
\rho(\mathbf 1_K\otimes b)=\omega(b),\qquad b\in \mathcal B(H), 
$$
and for every vector $\zeta$ in the range $R$ of the density operator of $\rho$, there 
is a unique operator $v\in\mathcal B(K_0,K)$ such that $(v\otimes\mathbf 1_H)\xi_\omega=\zeta$.    
Moreover, the natural map $v\mapsto (v\otimes\mathbf 1_H)\xi_\omega$ from the operator space 
$$
\mathcal S=\{v\in\mathcal B(K_0,K): (v\otimes\mathbf 1_H)\xi_\omega\in R\}
$$
to $R$ defines an isomorphism of complex 
vector spaces $\mathcal S\cong R$.  In particular, $\rank\rho=\dim\mathcal S$.  
\end{prop}

\begin{proof}
For existence of the operator $v$, we claim first that for every $b\in\mathcal B(H)$, 
$$
(\mathbf 1_{K_0}\otimes b)\xi_\omega=0\implies (\mathbf 1_K\otimes b)\zeta=0.  
$$ 
Indeed, if $(\mathbf 1_{K_0}\otimes b)\xi_\omega=0$ then $\omega(b^*b)=\|(\mathbf 1_{K_0}\otimes b)\xi_\omega\|^2=0$, 
so that $bp=0$,  $p$ being the support projection of $\omega$.   Since $\zeta$ belongs to the range of the 
support projection $q$ of $\rho$ and since $q\leq \mathbf 1_K\otimes p$ by Lemma \ref{oeLem0}, it follows that 
$(\mathbf 1_K\otimes b)\zeta=(\mathbf 1_K\otimes b)(\mathbf 1_K\otimes p)\zeta=(\mathbf 1_K\otimes bp)\zeta=0$.  

Hence we can define an operator $\tilde v: K_0\otimes H\to K\otimes H$ by 
$$
\tilde v(\mathbf 1_{K_0}\otimes b)\xi_\omega=(\mathbf 1_K\otimes b)\zeta,\qquad b\in \mathcal B(H).  
$$
It is clear from its definition that $\tilde v(\mathbf 1_{K_0}\otimes b)=(\mathbf 1_K\otimes b)\tilde v$ for 
$b\in\mathcal B(H)$, so that $\tilde v$ admits a unique factorization $\tilde v=v\otimes \mathbf 1_H$ with 
$v\in\mathcal B(K_0,K)$,  in 
the sense that $\tilde v(\xi\otimes \eta)=v\xi\otimes \eta$, for $\xi\in K_0$, $\eta\in H$.  

Uniqueness of $v$ is a straightforward consequence of 
the fact that $\xi_\omega$ is cyclic for the algebra $\mathbf 1_{K_0}\otimes\mathcal B(H)$.     
Finally, the last sentence 
is apparent from these assertions, since $v\mapsto (v\otimes \mathbf 1_H)\xi_\omega\in K\otimes H$ is a linear map.  
\end{proof}

Proposition \ref{oeProp1} leads to the following useful operator-theoretic criterion for separability.  
While it does not characterize the property,   we will give   
an operator-theoretic characterization of 
separability later in Proposition \ref{psProp2}.   

\begin{cor}\label{oeCor1}
Let $\omega$, $\xi_\omega$, $\rho$, $R$ and 
$$
\mathcal S=\{v\in\mathcal B(K_0,K): (v\otimes\mathbf 1_H)\xi_\omega\in R\}
$$ 
be as in Proposition \ref{oeProp1}.  Let $w\in \mathcal S$ and let $\zeta=(w\otimes \mathbf 1)\xi_\omega$.  
Then  $\zeta$ has the form $\zeta=\xi\otimes\eta$ for 
vectors $\xi\in K$, $\eta\in H$ iff $\rank(w)\leq 1$.   If $\rho$ is a separable state, then the operator space 
$\mathcal S$ has a basis consisting of rank-one operators.  
\end{cor}

\begin{proof}
Fix $w\in\mathcal S$ and assume first that $(w\otimes\mathbf 1)\xi_\omega$ decomposes into 
a tensor product $\xi\otimes \eta$ 
for vectors $\xi\in K$, $\eta\in H$. We use the fact that 
$\xi_\omega$ is cyclic for $\mathbf 1_{K_0}\otimes\mathcal B(H)$ to write 
\begin{align*}
wK_0\otimes H&=(w\otimes\mathbf 1_H)(\mathbf 1_{K_0}\otimes\mathcal B(H))\xi_\omega
=(\mathbf 1_K\otimes\mathcal B(H))(w\otimes\mathbf 1_H)\xi_\omega
\\
&=\xi\otimes\mathcal B(H)\eta=\xi\otimes H.  
\end{align*}
It follows that $wK_0= \mathbb C\cdot\xi$, as asserted.  Conversely, if $wK_0=\mathbb C\cdot\xi$ 
for some $\xi\in K$, then $(w\otimes\mathbf 1)\xi_\omega\in (w\otimes\mathbf 1)(K\otimes H)\subseteq\xi\otimes H$, 
hence there is a vector 
$\eta\in H$ such that $(w\otimes\mathbf 1)\xi_\omega=\xi\otimes\eta$.  

 If $\rho$ is separable, then it can be 
written as a convex combination of pure separable states of $\mathcal B(K\otimes H)$, and 
this implies that $R$ is spanned by vectors of the form $\xi\otimes \eta$, with $\xi\in K$ and 
$\eta\in H$ (this is known as the {\em range criterion} for separability in the physics literature).  Hence 
there is a linear basis for $R$ consisting of vectors of the form  
$\xi_k\otimes\eta_k$, $k=1,\dots,r$.  
By Proposition \ref{oeProp1}, there are operators $w_1,\dots,w_r\in \mathcal B(K_0,K)$ such that 
$(w_k\otimes\mathbf 1_H)\xi_\omega=\xi_k\otimes\eta_k$, and  Proposition \ref{oeProp1} 
also implies that $w_1,\dots,w_r$ is a linear basis for the operator space $\mathcal S$.  The 
paragraph above implies $\rank w_k\leq 1$ for all $k$.  
\end{proof}

\section{Sums of positive rank-one operators}\label{S:ro}

We require 
the following description of the possible ways a positive finite rank 
operator $A$ can be represented as a sum of positive rank one operators 
$$
A=\xi_1\otimes\bar\xi_1+\cdots+\xi_r\otimes\bar\xi_r.   
$$ 
Significantly, the vectors 
$\xi_1,\dots,\xi_r$ involved in this representation of the operator 
$A$ need not be linearly independent - nor even nonzero - and 
that flexibility is essential for our purposes.  For completeness, we include 
a proof of this bit of the lore of elementary operator theory.

\begin{prop}\label{roProp1}
Let $\xi_1,\dots,\xi_r$ and $\eta_1,\dots,\eta_r$ be two $r$-tuples of vectors in a Hilbert 
space $H$.  Then 
\begin{equation}\label{ssEq3}
\xi_1\otimes\bar\xi_1+\cdots+\xi_r\otimes \bar\xi_r=\eta_1\otimes\bar\eta_1+\cdots+\eta_r\otimes\bar\eta_r.
\end{equation}
iff there is a unitary $r\times r$ matrix $(\lambda_{ij})$ of complex numbers such that 
\begin{equation}\label{ssEq4}
\eta_i=\sum_{j=1}^r\lambda_{ij}\xi_j,\quad \xi_i=\sum_{j=1}^r\bar\lambda_{ji}\eta_j, \qquad 1\leq i\leq r.  
\end{equation}
\end{prop}

\begin{proof}  In the statement of Proposition \ref{roProp1}, the notation $\xi\otimes\bar\xi$ denotes the operator 
$\zeta\mapsto \langle \zeta,\xi\rangle \xi$.  In order to show that (\ref{ssEq3}) implies (\ref{ssEq4}),   
consider the two operators $A,B: \mathbb C^r\to H$ defined by 
$$
A(\lambda_1,\dots,\lambda_r)=\sum_k\lambda_k\xi_k, \qquad 
B(\lambda_1,\dots,\lambda_r)=\sum_k\lambda_k\eta_k.
$$ 
The adjoint of $A$ is given by $A^*\zeta=(\langle \zeta,\xi_1\rangle,\dots,\langle \zeta,\xi_r\rangle)$, 
with a similar formula for $B^*$, and the hypothesis (\ref{ssEq3}) becomes $AA^*=BB^*$.  It follows that 
$\|A^*\zeta\|=\|B^*\zeta\|$ for all $\zeta\in H$, and 
we can define a partial isometry $w_0$ with initial space $A^*H$ and final space $B^*H$ by setting 
$w_0(A^*\zeta)=B^*\zeta$, $\zeta\in H$.  Since $\mathbb C^r$ is finite-dimensional, $w_0$ 
can be extended to a unitary operator $w\in\mathcal B(\mathbb C^r)$, and we have $B=Aw^{-1}$.  Letting $e_1,\dots,e_r$ be 
the usual basis for $\mathbb C^r$,  we find that the matrix $(\lambda_{ij})$ of $w^{-1}$ 
relative to $(e_k)$ satisfies 
$$
\eta_i=Be_i=Aw^{-1}e_i=\sum_{j=1}^r \lambda_{ij}Ae_j=\sum_{j=1}^r\lambda_{ij}\xi_j.  
$$ 
The second formula of (\ref{ssEq4}) follows from the line above after substituting these formulas for 
$\eta_k$ in $\sum_k\bar\lambda_{ki}\eta_k$ and using unitarity of the matrix $(\lambda_{ij})$.  

The converse is a straightforward calculation using unitarity of the matrix $(\lambda_{ij})$ 
that we omit.  
\end{proof}

\section{Parameterizing the extensions of a state}\label{S:ss}

Let $H$, $K$ be Hilbert spaces satisfying $n=\dim H\leq m=\dim K<\infty$.  
Given a state $\omega$ of $\mathcal B(H)$, we consider the compact convex set $E(\omega)$ 
of all extensions of $\omega$ to a state of $\mathcal B(K\otimes H)$.  
Remark \ref{ssRem1} shows that without loss of generality, we can restrict attention to 
the case in which $\omega$ is a {\em faithful} state of $\mathcal B(H)$, and we do so.

Consider the filtration of $E(\omega)$ into compact subspaces 
$$
E^1(\omega)\subseteq E^2(\omega)\subseteq \cdots\subseteq E^{mn}(\omega)=E(\omega),  
$$
where $E^r(\omega)$ denotes the space of all states of $E(\omega)$ satisfying 
$\rank \rho\leq r$.  The spaces $E^r(\omega)$ are no longer convex; but since 
$\dim K\geq \dim H$, one can exhibit pure states in $E(\omega)$ -  for example, 
the state $\rho(x)=\langle x\zeta,\zeta\rangle$, where $\zeta$ is 
a unit vector in $K\otimes H$ of the form 
\begin{equation}\label{ssEq0}
\zeta=\sqrt{\lambda_1}\cdot f_1\otimes e_1+\cdots +\sqrt{\lambda_n}\cdot f_n\otimes e_n
\end{equation}
where $e_1,\dots,e_n$ is an orthonormal basis for $H$ consisting of eigenvectors of the density operator 
of $\omega$ with $\lambda_1,\dots,\lambda_n$ the corresponding eigenvalues, and 
where $f_1,\dots,f_n$ is an arbitrary orthonormal set in $K$.   In particular, 
the spaces $E^r(\omega)$ are nonempty for every $r\geq 1$.

Now fix an integer $r$ in the range $1\leq r\leq mn$.  We define a map from 
the noncommutative sphere $V^r(H,K)$ to $E^r(\omega)$ as follows.  
Since $\omega$ is faithful, Lemma \ref{oeLem1} implies that there is a vector $\xi_\omega\in H\otimes H$ 
such that 
\begin{equation}\label{ssEq1}
{\rm{span\,}}(\mathbf 1_H\otimes\mathcal B(H))\xi_\omega=H\otimes H, 
\qquad \omega(b)=\langle (\mathbf 1\otimes b)\xi_\omega,\xi_\omega\rangle, \quad b\in N. 
\end{equation}
Choose an $r$-tuple $v=(v_1,\dots,v_r)\in V^r(H,K)$.   
Since each $v_k\otimes \mathbf 1_H$ maps $H\otimes H$ to $K\otimes H$, we can 
define a linear functional $\rho_{v}$ on $\mathcal B(K\otimes H)$ as follows 
\begin{equation}\label{ssEq2}
\rho_{v}(x)=\sum_{k=1}^r\langle x(v_k\otimes\mathbf 1_H)\xi_\omega,(v_k\otimes\mathbf 1_H)\xi_\omega\rangle, 
\qquad x\in\mathcal B(K\otimes H).  
\end{equation}
Clearly $\rho_v$ is positive, and since $v_1^*v_1+\cdots+ v_r^*v_r=\mathbf 1_H$, we have 
\begin{equation*}
\rho_v(\mathbf 1_K\otimes b)=\sum_{k=1}^r\langle (v_k^*v_k\otimes b)\xi_\omega,\xi_\omega\rangle 
=\langle (\mathbf 1_H\otimes b)\xi_\omega,\xi_\omega\rangle=\omega(b),    
\end{equation*}
for all $b\in \mathcal B(H)$.  It is obvious that the rank of $\rho_v$ cannot exceed $r$, hence 
$\rho_v\in E^r(\omega)$.  The purpose of this section is to prove:

\begin{thm}\label{ssThm1}  Let $H$, $K$ be Hilbert spaces of respective dimensions 
$n\leq m$, let $\omega$ be a faithful state of $\mathcal B(H)$, fix a vector $\xi_\omega\in H\otimes H$ 
as in (\ref{ssEq1}), and define a map 
$$
v\in V^r(H,K)\mapsto \rho_{v}\in E^r(\omega)
$$
as in (\ref{ssEq2}).  
Then $\rho_{v}=\rho_{v^\prime}$ iff there is 
an $r\times r$ unitary matrix of scalars $\lambda\in U(r)$ such that $v^\prime=\lambda\cdot v$.  
Moreover, for every $r=1,2,\dots,mn$, this map 
is a continuous surjection that maps open subsets of $V^r(H,K)$ 
to relatively open subsets of $E^r(\omega)$.  

If $\xi_\omega^\prime\in H$ is another vector satisfying (\ref{ssEq1}), giving rise to another  map 
$$
v\in V^r(H,K)\mapsto \rho^\prime_v\in E^r(\omega), 
$$ 
then there is a unitary operator $w\in \mathcal B(H)$ 
satisfying $\rho_v^\prime=\rho_{v\cdot w}$ for all $v$, where $(v_1,\dots,v_r)\cdot w=(v_1w,\dots,v_rw)$ 
denotes the right action of $w\in\mathcal U(H)$  on $v=(v_1,\dots,v_r)\in V^r(H,K)$.  
\end{thm}

\begin{proof}[Proof of Theorem \ref{ssThm1}]
Let $v=(v_1,\dots,v_r)$ 
and $v^\prime=(v_1^\prime,\dots,v_r^\prime)$ belong to $V_r(H,K)$, and assume first that $\rho_v=\rho_{v^\prime}$.  
Define vectors $\xi_k, \xi^\prime_k\in K\otimes H$ by $\xi_k=(v_k\otimes\mathbf 1_H)\xi_\omega$, 
$\xi^\prime_k=(v_k^\prime\otimes\mathbf 1_H)\xi_\omega$, $k=1,\dots,r$.  The density operators of 
$\rho_v$ and $\rho_{v^\prime}$ are 
$$
\sum_{k=1}^r\xi_k\otimes\bar\xi_k, \quad {\rm{and}}\quad \sum_{k=1}^r\xi_k^\prime\otimes \bar\xi_k^\prime 
$$
respectively, so that the hypothesis $\rho_v=\rho_{v^\prime}$ 
is equivalent to the assertion 
$$
\sum_{k=1}^r\xi_k\otimes\bar\xi_k=\sum_{k=1}^r\xi_k^\prime\otimes \bar\xi_k^\prime.
$$ 
By Proposition \ref{roProp1}, there is a unitary $r\times r$ matrix $(\lambda_{ij})$ of scalars such that 
$$
\xi_i^\prime=\sum_{j=1}^r \lambda_{ij}\xi_j, \qquad 1\leq i\leq r.     
$$
Proposition \ref{oeProp1} implies that 
$v_i^\prime=\sum_j\lambda_{ij}v_j$, $1\leq i\leq r$, hence $v^\prime=\lambda\cdot r$.  

Conversely, suppose there is a unitary matrix $\lambda=(\lambda_{ij})\in M_r(\mathbb C)$ such 
that  
$v^\prime=\lambda\cdot v$, and consider the vectors in $K\otimes H$ defined 
by $\xi_k=(v_k\otimes\mathbf 1_H)\xi_\omega$, $\xi_k^\prime=(v_k^\prime\otimes\mathbf 1_K)\xi_\omega$, 
$1\leq k\leq r$.  The relation $v^\prime=\lambda\cdot v$ implies that  
\begin{equation}\label{ssEq5}
\xi_i^\prime=\sum_{j=1}^r\lambda_{ij}\xi_j, 
\end{equation}
and the density operators of $\rho_v$ and $\rho_{v^\prime}$ are given 
respectively by 
$$
\sum_{k=1}^r \xi_k\otimes\bar\xi_k,\qquad \sum_{k=1}^r \xi_k^\prime\otimes\bar\xi_k^\prime.  
$$
Substitution of (\ref{ssEq5}) into the term on the right gives 
$$
\sum_{k=1}^r \xi_k^\prime\otimes\bar\xi_k^\prime=\sum_{k,p,q=1}^r \lambda_{kp}\bar\lambda_{kq}\xi_p\otimes\bar\xi_q.  
$$
Since $(\lambda_{ij})$ is a unitary matrix, this implies 
$\sum_k\xi_k^\prime\otimes\bar\xi_k^\prime=\sum_p\xi_p\otimes\bar\xi_p$, and 
$\rho_{v^\prime}=\rho_v$ follows.    

The preceding paragraphs imply that 
the mapping $v\mapsto \rho_v$ factors through the quotient $X^r=V^r(H,K)/U(r)$
$$
v\in V^r(H,K) \to \dot v\in X^r \to \rho_v,   
$$ 
and that the second map $\dot v\mapsto \rho_v$ is continuous and injective.  Hence 
it is a homeomorphism of $X^r$ onto its range, and the composite 
map $v\mapsto \rho_v$ is continuous and maps open sets to 
relatively open subsets of its range.  

It remains to show that every state of $E^r(\omega)$ belongs to the range of 
$v\mapsto \rho_v$.  Choose $\rho\in E^r(\omega)$.  Since the rank of $\rho$ is at most 
$r$ we can write it in the form 
\begin{equation}\label{ssEq6}
\rho(x)=\sum_{k=1}^r\langle x\zeta_k,\zeta_k\rangle,\qquad x\in\mathcal B(K\otimes H), 
\end{equation}
where the $\zeta_k$ are vectors in $K\otimes H$, perhaps with some being zero.

By Proposition \ref{oeProp1}, there are operators $v_1,\dots,v_r\in\mathcal B(H,K)$ such that 
$\zeta_k=(v_k\otimes\mathbf 1_H)\xi_\omega$ for each $k$, and we claim that 
$\sum_kv_k^*v_k=\mathbf 1_H$.  Indeed, for all $b_1,b_2\in \mathcal B(H)$ we have 
\begin{align*}
\langle (\sum_k v_k^*v_k)\otimes b_1)\xi_\omega,&(\mathbf 1_H\otimes b_2)\xi_\omega\rangle =
\sum_k\langle (v_k\otimes b_2^*b_1)\xi_\omega,(v_k\otimes \mathbf 1_H)\xi_\omega\rangle
\\
&=\sum_k\langle (\mathbf 1_K\otimes b_2^*b_1)\zeta_k,\zeta_k\rangle=\rho(\mathbf 1_K\otimes b_2^*b_1)
\\
&=\omega(b_2^*b_1)=\langle (\mathbf 1_H\otimes b_1)\xi_\omega,(\mathbf 1_H\otimes b_2)\xi_\omega\rangle,   
\end{align*}
and $\sum_kv_k^*v_k=\mathbf 1_H$ follows from cyclicity: 
$H\otimes H=(\mathbf 1_H\otimes \mathcal B(H))\xi_\omega$.  

Substituting back into (\ref{ssEq6}), we see that $v=(v_1,\dots,v_r)\in V^r(H,K)$ has been exhibited 
with the property $\rho=\rho_v$.  

To prove the last paragraph, choose another $\xi_\omega^\prime\in H$ satisfying (\ref{ssEq1}).  Then we have 
$\|(\mathbf 1\otimes b)\xi_\omega\|^2=\omega(b^*b)=\|(\mathbf 1\otimes b)\xi_\omega^\prime\|$ for 
every $b\in \mathcal B(H)$, hence there is a unique unitary operator in the commutant of $\mathbf 1\otimes\mathcal B(H)$
that maps $\xi_\omega$ to $\xi_\omega^\prime$.  Such an operator has the form $w\otimes \mathbf 1$ for a unique 
unitary operator $w\in \mathcal B(H)$, hence $\xi_\omega^\prime=(w\otimes \mathbf 1)\xi_\omega$.  
From the definition of the map (\ref{ssEq2}), it follows that 
the corresponding state $\rho_v^\prime$ is defined on $x\in\mathcal B(K\otimes H)$ by 
$$
\rho^\prime_v(x)=\sum_{k=1}^r\langle x(v_k\otimes\mathbf 1)\xi_\omega^\prime,(v_k\otimes\mathbf 1)\xi_\omega^\prime\rangle=
\sum_{k=1}^r\langle x(v_kw\otimes\mathbf 1)\xi_\omega,(v_kw\otimes \mathbf 1)\xi_\omega\rangle, 
$$
and the right side is seen to be $\rho_{v\cdot w}(x)$.  
\end{proof}

\section{The role of $(X^r,P^r)$ in entanglement}\label{S:ps}

In this section we give an operator-theoretic characterization of 
separable states and show that the probability of entanglement 
is positive at all levels (see Theorem \ref{psThm1}).    

Assume that $n=\dim H\leq m=\dim K<\infty$, fix $r=1,2,\dots, mn$,  
choose a faithful state $\omega$ of $\mathcal B(H)$, and choose a vector $\xi_\omega$ 
as in (\ref{ssEq1}).  
Theorem \ref{ssThm1} implies that the parameterizing 
map $v\in V^r(H,K)\mapsto \rho_v\in E^r(\omega)$ decomposes naturally into a 
composition of two maps 
\begin{equation}\label{psEq1}
v\in V^r(H,K)\mapsto \dot v\in X^r\mapsto \rho_v\in E^r(\omega).  
\end{equation}

We can promote 
the invariant probability measure $\mu$ on $V^r(H,K)$ all the way to $E^r(\omega)$ by way 
of the composite map 
$$
v\in V^r(H,K)\mapsto \rho_v\in E^r(\omega)
$$ 
thereby obtaining a compact metrizable probability space $(E^r(\omega),P^{r,\omega})$.  
\begin{rem}[Independence of the choice of $\omega$]\label{psRem0}
After noting that the second map 
of (\ref{psEq1}) implements a measure-preserving 
homeomorphism of topological probability spaces 
$(X^r,P^r)\cong (E^r(\omega),P^{r,\omega})$, we conclude that 
{\em each of the probability spaces $(E^r(\omega),P^{r,\omega})$ associated with 
faithful states of $\mathcal B(H)$ is 
isomorphic to the intrinsic space $(X^r,P^r)$, hence they are all isomorphic to each other.}
\end{rem}

\begin{rem}[Independence of the choice of $\xi_\omega$]
If we choose another vector $\xi_\omega^\prime\in H$ satisfying (\ref{ssEq1}), the 
resulting parameterization $v\mapsto\rho_v^\prime$ of $E^r(\omega)$ differs from that of (\ref{psEq1}), 
hence the resulting 
probability measure $P^{r,\omega \prime}$ on $E^r(\omega)$ 
appears to differ from the one $P^{r,\omega}$ promoted through the map $v\mapsto \rho_v$.   
However, Theorem \ref{ssThm1} implies that 
there is a unitary operator $w\in\mathcal U(H)$ such that 
$\rho_v^\prime=\rho_{v\cdot w}$,  $v\in V^r(H,K)$, so that 
$P^{r,\omega}$ and $P^{r,\omega \prime}$  are respectively promotions (through the same map 
$v\mapsto \rho_v$)
of the measure $P^r$ and its transform $P^{r \prime}$ under the right action of $w$ on $X^r$.  
Remark \ref{upsRem1} implies that $P^{r \prime}=P^r$, hence $P^{r,\omega \prime}=P^{r,\omega}$, 
and therefore $(E^r(\omega),P^{r,\omega})$ does not depend on the choice of $\xi_\omega$.   
\end{rem}

\begin{rem}[Invariance of rank and separability]\label{psRem1}
It is not obvious that spatial properties of states such as rank and separability are 
preserved under these identifications.  For example, it is not clear that the integer-valued 
random variable that represents rank on the probability space $(E^r(\omega),P^{r,\omega})$
$$
\rho\in E^r(\omega)\mapsto \rank\rho\in \{1,2,\dots,r\}  
$$
is preserved under the isomorphism 
$(E^r(\omega_1), P^{r,\omega_1})\cong (E^r(\omega_2),P^{r,\omega_2})$ for different faithful states $\omega_1$ and 
$\omega_2$.  Similarly, we require that these identifications should 
preserve separability and entanglement.  We establish the invariance of these properties 
in Propositions \ref{psProp1} and \ref{psProp2} below 
by identifying them appropriately 
in terms of random variables on 
the intrinsic probability space $(X^r,P^r)$.  
\end{rem}

We first establish the invariance of rank.  

\begin{prop}\label{psProp1}
Let $\omega$ be a faithful state of $\mathcal B(H)$, fix 
$r=1,2,\dots,mn$  and consider the factorization (\ref{psEq1}) through $X^r$ of the 
parameterization map $v\mapsto \rho_v$.  For every $v\in V^r(H,K)$, one has   
\begin{equation}\label{psEq2}
\rank(\dot v)=\rank\rho_v,   
\end{equation}
and almost surely, states of $(E^r(\omega),P^{r,\omega})$ have rank $r$. 
\end{prop}

\begin{proof}  Formula (\ref{psEq2}) simply restates the last sentence of Proposition \ref{oeProp1}, 
and the second phrase follows from Theorem \ref{upsThm1}.
\end{proof}

In order to establish a similar invariance result for the 
probability of entanglement/separability of states, we require 
an operator-theoretic characterization of separability (Proposition \ref{psProp2}).  
In turn, that requires a known upper bound that we collect 
in the following Lemma.

\begin{lem}\label{psLem1}
Every separable state of $\mathcal B(K\otimes H)$ is a convex combination of at most $m^2n^2$ pure 
separable states.  
\end{lem}

\begin{proof}
A straightforward application of Remark \ref{enRem1}.  
\end{proof}

Throughout the remainder of this section, we set $q=m^2n^2$ and let $U(q)$ be 
group of all $q\times q$ unitary matrices $\mu=(\mu_{ij})\in M_q(\mathbb C)$.  

\begin{prop}\label{psProp2}
Let $\omega$ be a faithful state of $\mathcal B(H)$, let $\rho\in E^r(\omega)$, and 
choose $v\in V^r(H,K)$ such that $\rho=\rho_v$.   Then $\rho$ is separable iff there is 
a unitary matrix $\mu=(\mu_{ij})$ in $U(q)$ such that 
\begin{equation}\label{psEq3}
\rank(\sum_{j=1}^r\mu_{ij}v_j)\leq 1,\qquad i=1,2,\dots,q.  
\end{equation}
\end{prop}

\begin{proof}
Assume first that $\rho$ is separable.  By Lemma \ref{psLem1}, there are vectors 
$\xi_i\in K$, $\eta_i\in H$, $1\leq i\leq q$, such that 
\begin{equation*}
\rho(x)=\sum_{i=1}^q \langle x(\xi_i\otimes \eta_i),\xi_i\otimes \eta_i\rangle, \qquad x\in\mathcal B(K\otimes H).  
\end{equation*}
Let $v_i^\prime=v_i$ if $1\leq i\leq r$, set $v_i^\prime=0$ for $r<i\leq q$ and choose 
a vector $\xi_\omega\in H\otimes H$ that represents $\omega(b)=\langle(\mathbf 1\otimes b)\xi_\omega,\xi_\omega\rangle$ 
as in Lemma \ref{oeLem1}.  Then the formula 
$\rho=\rho_v$ can be rewritten 
$$
\rho(x)=\sum_{i=1}^q\langle x(v_i^\prime\otimes \mathbf 1)\xi_\omega,(v_i^\prime\otimes \mathbf 1)\xi_\omega\rangle, 
\qquad x\in\mathcal B(K\otimes H).  
$$
By Proposition  \ref{roProp1}, there is a unitary $q\times q$ matrix $\lambda=(\lambda_{ij})$ such that 
\begin{equation}\label{psEq4}
\xi_i\otimes\eta_i=\sum_{j=1}^q \lambda_{ij}(v_j^\prime\otimes\mathbf 1)\xi_\omega=
(\sum_{j=1}^r\lambda_{ij}v_j\otimes\mathbf 1)\xi_\omega,
\quad i=1,\dots,q.  
\end{equation}
Proposition \ref{oeProp1} implies that for every $i=1,\dots,q$ there is a unique operator 
$w_i:H\to K$ such that $(w_i\otimes\mathbf 1)\xi_\omega=\xi_i\otimes \eta_i$, and  
(\ref{psEq4}) plus uniqueness implies 
\begin{equation*}
w_i=\sum_{j=1}^r \lambda_{ij}v_j, \qquad i=1,2,\dots,q.  
\end{equation*}
Finally, Corollary \ref{oeCor1} implies that $w_i$ is of rank at most $1$, 
and (\ref{psEq3}) follows.  

All of these steps are reversible, and we leave the proof of the converse 
assertion for the reader.  
\end{proof}

We can now identify the subsets of $X^r$ that correspond to separable or 
entangled extensions of faithful states of $\mathcal B(H)$.   

\begin{prop}\label{psProp3}
For every $r=1,2,\dots, mn$, 
let $\Sep(V^r(H,K))$ be the subset of $V^r(H,K)$ defined by the conditions of (\ref{psEq3})
$$
\Sep(V^r(H,K))=\{v: \exists\ \mu\in U(q)\  {\rm{s.\ t.\ }} \rank(\sum_{j=1}^r\mu_{ij}v_j)\leq 1, \quad  1\leq i\leq q\}.
$$  
The natural projection $v\mapsto \dot v$ of $V^r(H,K)$ on $X^r$ carries  
$\Sep(V^r(H,K))$ onto a closed subset 
$\Sep(X^r)$ of $X^r$ that is invariant under the right action of $\mathcal U(H)$, 
and which has the following properties:  For every 
faithful state $\omega$ of $\mathcal B(H)$ and every $v\in V^r(H,K)$
\begin{enumerate}
\item[(i)]
$\rho_v$ is a separable state of $E^r(\omega)$  iff  $\dot v\in \Sep(X^r)$.  
\item[(ii)] $\rho_v$ is an entangled state of $E^r(\omega)$ iff  $\dot v\in X^r\setminus \Sep(X^r)$. 
\end{enumerate}
\end{prop}

\begin{proof}
For a fixed faithful state $\omega$ of $\mathcal B(H)$, Proposition \ref{psProp2} implies that 
the homeomorphism $\dot v\mapsto \rho_v$ maps $\Sep(X^r)$ 
onto the space of separable states in $E^r(\omega)$.  Since 
the separable states form a closed subset of the state space of $\mathcal  B(K\otimes H)$, it follows that 
$\Sep(X^r)$ is closed.  Invariance under the right action of $\mathcal U(H)$ on $X^r$ 
follows from the fact that 
for every operator $v\in \mathcal B(H,K)$ and every unitary operator $w$ on $H$, 
$\rank(vw)=\rank(v)$.  
Assertion (i) is a restatement of Proposition \ref{psProp2}, and (ii) follows from (i) since entangled 
states and separable states are complementary sets.  
\end{proof}

The following result implies that there are plenty of entangled states of all possible ranks.  
We will obtain sharper results in Sections \ref{S:er} and \ref{S:di}.  

\begin{thm}\label{psThm1}
For every $r=1,2,\dots,mn$, $\Sep(X^r)$ is a proper closed subset of $X^r$, 
and for every faithful state $\omega$ of $\mathcal B(H)$, the probability $p$ 
of entanglement in $(E^r(\omega),P^{r,\omega})$ is independent of the choice of $\omega$ 
and satisfies 
$$
p=1-P^r(\Sep(X^r))=P^r(X^r\setminus \Sep(X^r))>0.
$$  
\end{thm}

\begin{proof}  Fix $r=1,2,\dots,mn$.  We claim first that there is a faithful 
state $\omega$ of $\mathcal B(H)$ such that $E^r(\omega)$ 
contains an entangled state.  To see that, choose an orthonormal basis $e_1,\dots,e_n$ for 
$H$, an orthonormal set $f_1,\dots,f_n\in K$, and 
let $\zeta$ be the unit vector 
$$
\zeta=\frac{1}{\sqrt n}(f_1\otimes e_1+\cdots+f_n\otimes e_n)\in K\otimes H.  
$$
It is well known that $\rho(x)=\langle x\zeta,\zeta\rangle$, $x\in\mathcal B(K\otimes H)$, defines a pure 
entangled state of $\mathcal B(K\otimes H)$ that restricts to the tracial state on $\mathcal B(H)$.  

We claim that there is a self-adjoint operator $c\in \mathcal B(K\otimes H)$ 
such that $\rho(c)<0$ and such that for all states $\sigma_1$ of $\mathcal B(K)$ 
and $\sigma_2$ of $\mathcal B(H)$, one has 
\begin{equation}\label{psEq5}
(\sigma_1\otimes\sigma_2)(c)\geq 0.  
\end{equation}
Indeed, since $\zeta$ is not a tensor product, we have $|\langle \xi\otimes \eta,\zeta\rangle|<1$ 
for every pair of unit vectors $\xi\in K$, $\eta\in H$; and    
since the unit spheres of $K$ and $H$ are compact, we can choose $\alpha\in(0,1)$ such that 
$$
\max\{|\langle \xi\otimes \eta,\zeta\rangle|^2: \xi\in K,\ \eta\in H,\ \|\xi\|=\|\eta\|=1\}\leq\alpha<1.     
$$
Set $c=\alpha\cdot\mathbf 1-\zeta\otimes\bar\zeta$.  Obviously $\rho(c)<0$, and by its construction, 
$c$ satisfies (\ref{psEq5}) for pure states $\sigma_1$ and $\sigma_2$.  (\ref{psEq5}) follows in 
general, since every 
state is a convex combination of pure states.

Now choose any projection $p$ of rank $r$ in $\mathcal B(K\otimes H)$ 
whose range contains $\zeta$.  Then for every $t\in (0,1)$, 
$$
\sigma_t(x)=\frac{t}{r}\tr(px)+(1-t)\cdot \rho(x),\qquad x\in\mathcal B(K\otimes H)
$$
is a state of rank $r$ that restricts to a faithful state $\omega_t$ of 
$\mathcal B(H)$.  Moreover, for sufficiently small $t$, we will have $\sigma_t(c)<0$; and    
for such $t$ (\ref{psEq5}) implies that $\sigma_t$ is not a convex combination 
of product states, proving the claim.    

Choose a faithful state $\omega$ of $\mathcal B(H)$ such that $E^r(\omega)$ contains 
an entangled state $\rho_0$.  Then the inverse image $x_0\in X^r$ of $\rho_0$ under the 
map $\dot v\in X^r\mapsto \rho_v\in E^r(\omega)$ is a point in the complement of $\Sep(X^r)$, hence $\Sep(X^r)\neq X^r$.  The set 
$X^r\setminus \Sep(X^r)$ is a nonempty open subset of $X^r$ which therefore has positive 
$P^r$-measure.  It follows from Proposition \ref{psProp3} 
that the probability $p$ of entanglement in $(E^r,P^{r,\omega})$ 
satisfies $p=P^r(X\setminus\Sep(X^r))>0$.  Finally, Proposition \ref{psProp3} and 
Remark \ref{psRem0} imply that the same assertions are true for the 
probability space $(E^r(\omega^\prime), P^{r,\omega^\prime})$ associated with any faithful 
state $\omega^\prime$ of $\mathcal B(H)$, and that the probability of entanglement in $(E(\omega^\prime),P^{r,\omega^\prime})$ 
does not depend on the choice of $\omega^\prime$.  
\end{proof}

\section{Properties of the wedge invariant}\label{S:wi}

Proposition \ref{psProp3} implies that among the states $\rho_v$ of $E^r(\omega)$, 
the separability property 
is determined by membership of $\dot v$ in the closed set $\Sep(X^r)$.  
Hence, in order to calculate or estimate the probability of entanglement 
in the spaces $(E^r(\omega),P^{r,\omega})$, one needs to calculate or estimate 
$P^r(\Sep(X^r))$.    Writing $q=m^2n^2$ as in the preceding section, 
the set $\Sep(X^r)$ is identified in Propositions \ref{psProp2} and \ref{psProp3} as 
\begin{equation}\label{wiEq1}
\Sep(X^r)=\bigcup_{\mu\in U(q)}\{\dot v\in X^r: \rank(\sum_{j=1}^r\mu_{ij}v_j)\leq 1,\quad 1\leq i\leq q\}.  
\end{equation}
The set on the right defines an uncountable union of subvarieties of $V^r(H,K)$, but it 
is not a subvariety itself nor even a countable union of subvarieties (see Section \ref{S:ec}).   In this section 
we reformulate the definition of the wedge invariant (Definition \ref{wsDef1}) as a pair of random variables 
$$
\dot w,\dot w^*: X^r\to \{0,1,2,\dots\}.
$$  
We show that these random variables provide a nontrivial test for separability -- i.e., membership 
in $\Sep(X^r)$ -- and  
that they define subvarieties 
$$
A=\{v\in V^r(H,K): \dot w(\dot v)\leq 1\}, \quad A^*=\{v\in V^r(H,K): \dot w^*(\dot v)\leq 1\}, 
$$ 
with the property that $\Sep(X^r)\subseteq \dot A\cap \dot A^*$.  
The latter property is critical for the applications of Section \ref{S:er}.

Fix $r=1,2,\dots,mn$ and choose $v=(v_1,\dots,v_r)\in V^r(H,K)$.  We can form the operator 
$v_1\wedge\cdots\wedge v_r\in\mathcal B(H^{\otimes r},K^{\otimes r})$ as in (\ref{inEq4}), 
and this operator maps the symmetric subspace of $H^{\otimes r}$ to the antisymmetric 
subspace of $K^{\otimes r}$.  
If $v$ and $v^\prime$ belong to the same $U(r)$-orbit, say $v^\prime=\lambda\cdot v$ with $\lambda=(\lambda_{ij})\in U(r)$, 
then by elementary multilinear algebra we have 
\begin{equation}\label{wiEq2}
v_1^\prime\wedge\cdots\wedge v_r^\prime=\det(\lambda_{ij})\cdot v_1\wedge\cdots\wedge v_r.  
\end{equation}
It follows that $v_1^\prime\wedge\cdots\wedge v_r^\prime(H_+^{\otimes r})=v_1\wedge\cdots\wedge v_r(H_+^{\otimes r})$.   
Similarly, we can form  $v_1^*\wedge\cdots\wedge v_r^*\in \mathcal B(K^{\otimes r}, H^{\otimes r})$, 
and $(v_1^*\wedge\cdots\wedge v_r^*)(K_+^{\otimes r})$ depends only on the $U(r)$ orbit of $v$.  
Thus we can define integer-valued random variables 
$\dot w, \dot w^*: X^r\to \{0,1,2,\dots\}$ by 
\begin{equation}\label{wiEq3}
\dot w(\dot v) =\rank (v_1\wedge\cdots\wedge v_r\restriction_{H_+^{\otimes r}}), \quad 
\dot w^*(\dot v) =\rank (v_1^*\wedge\cdots\wedge v_r^*\restriction_{K_+^{\otimes r}}), 
\end{equation}
for $v\in V^r(H,K)$.  
The following result implies that these random variables can detect entanglement.  
Note too that both random variables $\dot w$ and $\dot w^*$ are invariant under the 
right action of $\mathcal U(H)$ on $X^r$.  

\begin{prop}\label{wiProp1}  For every $x\in \Sep(X^r)$, we have $\dot w(x)\leq 1$ 
and $\dot w^*(x)\leq 1$.  
\end{prop}

\begin{proof}  We claim that $\dot w\leq 1$ on $\Sep(X^r)$.  
Indeed, every point of $\Sep(X^r)$ has the form $x=\dot v$, where $v=(v_1,\dots,v_r)$ is an $r$-tuple 
in $V^r(H,K)$ whose associated state $\rho_v$ is separable.  We have to show that the 
restriction of the operator 
$v_1\wedge\cdots\wedge v_r$ to the symmetric subspace $H_+^{\otimes r}$ has rank $\leq 1$.  

To see that, note that Corollary \ref{oeCor1} implies 
that there is a linearly independent set of operators $w_1,\dots,w_r\in\mathcal B(H,K)$ that 
has the same linear span as $v_1,\dots,v_r$, such that $\rank w_k=1$ for every $k$.  
Since $v_1,\dots,v_r$ and $w_1,\dots,w_r$ are 
linearly independent subsets of $\mathcal B(H,K)$ 
that have the same linear span $\mathcal S$, elementary 
multilinear algebra implies that there is a complex number $d\neq 0$ such that 
$$
v_1\wedge\cdots\wedge v_r=d\cdot w_1\wedge\cdots\wedge w_r; 
$$
indeed, $d$ is the determinant of the linear operator defined on $\mathcal S$ 
by stipulating that it should carry one basis to the other.   Hence it is 
enough to show that the restriction of $w_1\wedge\cdots\wedge w_r$ to $H_+^{\otimes r}$ 
has rank at most $1$.

For  
every vector $\zeta\in H$ we have 
$$
(w_1\wedge\cdots\wedge w_r)(\zeta^{\otimes r})=w_1\zeta\wedge w_2\zeta\wedge\cdots\wedge w_r\zeta.  
$$
Now since each $w_k$ is of rank at most 
$1$, for every $k$ there are vectors $\zeta_k\in H$ and $\xi_k\in K$  
such that $w_k\zeta_k=\xi_k$ and $w_k=0$ on $\{\zeta_k\}^\perp$.  
For each $k$ we can write $\zeta=\mu_k\zeta_k+\zeta_k^\prime$
where $\mu_k\in\mathbb C$ and $\zeta_k^\prime$ belongs to the kernel of $w_k$.  Hence  
the term on the right takes the form  
$$
w_1(\mu_1\zeta_1)\wedge w_2(\mu_2\zeta_2)\wedge\cdots\wedge w_r(\mu_r\zeta_r)=
(\mu_1\mu_2\cdots\mu_r)\cdot \xi_1\wedge \xi_2\wedge\cdots\wedge \xi_r,  
$$
so that  $(w_1\wedge\cdots\wedge w_r)(\zeta^{\otimes r})
\in\mathbb C\cdot\xi_1\wedge\xi_2\wedge\cdots\wedge \xi_r$.  
Finally, a standard polarization argument 
shows that the symmetric subspace of $H^{\otimes r}$ is spanned by vectors of the form 
$\zeta^{\otimes r}$ with $\zeta\in H$, and the desired assertion 
$$
(w_1\wedge\cdots\wedge w_r)(H^{\otimes r}_{+})\subseteq \mathbb C\cdot\xi_1\wedge\xi_2\wedge\cdots\wedge\xi_r
$$ 
follows.  

The proof that 
$$
\dot w^*(\dot v)=\rank(v_1^*\wedge\cdots\wedge v_r^*\restriction_{K_+^{\otimes r}})\leq 1
$$ 
is similar, 
since the operators $w_1^*,\dots,w_r^*$ form a basis for the operator space $\mathcal S^*$ consisting of 
rank-one operators.  
\end{proof}

We have already pointed out that the analysis of states of $\mathcal B(K\otimes H)$ 
can be reduced to 
the analysis of states that restrict to faithful states on $\mathcal B(H)$.  Hence the result 
stated in Theorem \ref{wsThm1} of the introduction follows from Proposition \ref{wiProp1} 
and the fact that for every 
faithful state $\omega$ of $\mathcal B(H)$ and every state $\rho\in E^r(\omega)$ for $r=1,2,\dots,mn$, we have 
\begin{equation}\label{wiEq4}
w(\rho_v)=\dot w(\dot v), \qquad w^*(\rho_v)=\dot w^*(\dot v), \qquad v\in V^r(H,K).  
\end{equation}

Most significantly, the wedge invariant is associated with subvarieties: 

\begin{prop}\label{wiProp2}
For every $r=1,2,\dots,mn$, let 
$$
A=\{v\in V^r(H,K): \dot w(\dot v)\leq 1\},\quad A^*=\{v\in V^r(H,K): \dot w^*(\dot v)\leq 1\}.  
$$ 
Then both $A$ and $A^*$ are subvarieties of $V^r(H,K)$.  
\end{prop}

\begin{proof}  The set $A$ consists of all $r$-tuples $v\in V^r(H,K)$ 
such that the operator 
$G(v)=v_1\wedge\cdots\wedge v_r\restriction_{H_+^{\otimes r}}\in \mathcal B(H_+^{\otimes r}, K_-^{\otimes r})$ 
satisfies $\rank G(v)\leq 1$, or equivalently, that $G(v)\wedge G(v)=0$, where 
$G(v)\wedge G(v)$ is now viewed as an operator 
from $H_+^{\otimes r}\wedge H_+^{\otimes r}$ to $K_-^{\otimes r}\wedge K_-^{\otimes r}$.  Hence 
$
F(v)=G(v)\wedge G(v)
$
is a homogeneous polynomial of degree $2r$ with the property 
$$
A=\{v\in V^r(H,K): F(v)=0\}, 
$$
thereby exhibiting $A$ as a subvariety.  
A similar argument with $v_k^*$ replacing $v_k$ shows that $A^*$ is a subvariety.  
\end{proof}

Propositions \ref{wiProp1} and \ref{wiProp2} provide no information as to whether the wedge invariant is nontrivial, 
but the following result does.

\begin{prop}\label{wiProp3}
Assume that $\dim K\geq \dim H\geq 2$.  Then for every integer $r$ satisfying 
$1\leq r\leq \dim H/2$ there is a point $x\in X^r$ such that $\rank x=r$ and 
$\dot w^*(x)>1$, and the following equivalent assertions are true:
\begin{enumerate}
\item[(i)] The subvariety $A^*$ of Proposition \ref{wiProp2} is proper; $A^*\neq V^r(H,K)$.
\item[(ii)] For every faithful state $\omega$ of $\mathcal B(H)$ there is 
a state of rank $r$ in $E^r(\omega)$ such that $w^*(\rho)>1$.  
\end{enumerate}
\end{prop}

\begin{proof}
It suffices to 
exhibit an $r$-tuple $v=(v_1,\dots,v_r)\in V^r(H,K)$ such that 
$\rank(v_1^*\wedge\cdots\wedge v_r^*\restriction_{K^{\otimes r}_+})>1$.  Since $v_1^*\wedge\cdots\wedge v_r^*\neq 0$, 
the operators $v_1^*,\dots,v_r^*$ are linearly independent, hence so are $v_1,\dots,v_r$.  Proposition 
\ref{oeProp1} will then imply that the associated 
state $\rho_v$ has rank $r$, and it will satisfy $w^*(\rho_v)>1$ because of the asserted 
properties of $v_1,\dots,v_r$.  

We exhibit such operators $v_1,\dots,v_r$ as follows.  
Write $\dim H=2r+s$ with $s\geq 0$ and choose an orthonormal basis for $H$, enumerated by  
$$
\{e_1,\dots,e_r, f_1,\dots,f_r\}, \quad{\text{or}}\quad\{e_1,\dots,e_r, f_1,\dots,f_r, g_1,\dots, g_s\},
$$ 
according as $s=0$ or $s>0$.  Let $\{e_i^\prime,f_j^\prime, g_k^\prime\}$ be a similarly 
labelled orthonormal set in $K$.  For each $k=1,\dots,r$, let $v_k$ be the unique operator 
in $\mathcal B(H,K)$ satisfying $v_k e_i=\delta_{ki}e_1^\prime$ and $v_kf_i=\delta_{ki}f_1^\prime$ for $1\leq i\leq r$ if 
$s=0$, and otherwise satisfies the additional conditions 
$v_1 g_j=g_j^\prime$ and $v_2g_j=\cdots=v_rg_j=0$ for $j=1,\dots,s$ when $s>0$.  Each 
$v_k$ is a partial isometry whose adjoint 
$v_k^*$ maps $e_i^\prime$ to $\delta_{ik}e_k$ and $f_i^\prime$ to $\delta_{ik}f_k$ for $1\leq k\leq r$.  It 
follows that $v_1^*v_1+\cdots+v_r^*v_r=\mathbf 1_H$, so that $v=(v_1,\dots,v_r)\in V^r(H,K)$.  

Now consider the operator $v_1^*\wedge\cdots\wedge v_r^*$, restricted to the symmetric subspace $K^{\otimes r}_+$ of $K^{\otimes r}$.  
We have 
$$
(v_1^*\wedge\cdots\wedge v_r^*)(e_1^\prime\otimes \cdots\otimes e_1^\prime)=
v_1^*e_1^\prime\wedge v_2^*e_1^\prime\wedge\cdots\wedge v_r^*e_1^\prime=
e_1\wedge e_2\wedge\cdots\wedge e_r, 
$$
and similarly $(v_1^*\wedge\cdots\wedge v_r^*)(f_1^\prime\otimes\cdots\otimes f_1^\prime)=f_1\wedge f_2\wedge\cdots\wedge f_r$.  
Since the vectors $e_1\wedge e_2\wedge\cdots\wedge e_r$ and $f_1\wedge f_2\wedge\cdots\wedge f_r$ are mutually 
orthogonal unit vectors in $\wedge^r H$, it follows that $\rank(v_1^*\wedge\cdots\wedge v_r^*\restriction_{K^{\otimes r}_+})\geq 2$.  
\end{proof}

\section{Entangled states of small rank}\label{S:er}

We now assemble the results of the previous section into a main result.  
Fix Hilbert spaces $H$, $K$ with $2\leq n=\dim H\leq m=\dim K<\infty$.

\begin{thm}\label{erThm1}
Let $r$ be a positive integer satisfying $1\leq r\leq n/2$, let $\omega$ be a 
faithful state of $\mathcal B(H)$, and let $(E^r(\omega), P^{r,\omega})$ be the
probability 
space of Section \ref{S:ps}.   
Then almost every state of $(E^r(\omega),P^{r,\omega})$ is entangled.  
\end{thm}

\begin{proof} By 
Theorem \ref{ssThm1} and Proposition \ref{wiProp1}, the set 
of separable states of $E^r(\omega)$ is a closed subset of 
$$
\{\rho_v: v\in V^r(H,K),\ w^*(\rho_v)\leq 1 \}, 
$$ 
hence it suffices to show that the set 
$A^*=\{v\in V^r(H,K): w^*(\rho_v)\leq 1\}$ has $\mu$-measure zero.  
But by Propositions \ref{wiProp2} and \ref{wiProp3}, $A^*$ is a proper subvariety 
of $V^r(H,K)$, so that 
$\mu(A^*)=0$ follows from Proposition \ref{nsProp1}.  
\end{proof}

\begin{rem}[The meaning of ``relatively small rank"]\label{erRem1}  
In somewhat more prosaic terms, Theorem \ref{erThm1} 
has the following consequence.  Let $\rho$ be an arbitrary state 
of $M_m(\mathbb C)\otimes M_n(\mathbb C)$ and let $\omega$ 
be its marginal $\omega(a)=\rho(\mathbf 1\otimes a)$, $a\in M_n(\mathbb C)$.  
Then {\em whenever the inequalities 
$2\cdot\rank\rho\leq \rank\omega\leq m$ are satisfied,  one can infer from 
Theorem \ref{erThm1} 
that $\rho$ is entangled}, or else one has made a statistically impossible choice of $\rho$ that 
cannot be reproduced. 
\end{rem}

\begin{rem}[States of very small rank]  We note that 
if $r<\sqrt n$ in the hypothesis of Theorem \ref{erThm1}, then {\em every} state 
of $E^r(\omega)$ is entangled - or equivalently, 
$\Sep X^r=\emptyset$.  To sketch the elementary proof of that fact, let $\rho$ be a 
separable state of $\mathcal B(K\otimes H)$ such that $\rank\rho=r$, with $n=\dim H\leq \dim K<\infty$, 
and let $R\subseteq K\otimes H$ be the $r$-dimensional 
range of the density operator of $\rho$.   Since 
$\rho$ is separable it has a representation 
$$
\rho=\sum_{k=1}^s p_k\cdot\omega_k
$$
in which the $p_k$ are positive numbers summing to $1$ and the $\omega_k$ are pure 
product states of $\mathcal B(K\otimes H)$.  Since each $p_k>0$, the 
vector $\xi_k\otimes\eta_k$ associated with each $\omega_k$ must belong to $R$, and we 
can view the above formula as a relation between states of $\mathcal B(R)$.  At this point, Caratheodory's 
theorem (see Remark \ref{enRem1}) implies that there is a subset 
$S\subseteq\{\xi_1\otimes \eta_1,\dots,\xi_s\otimes \eta_s\}\subseteq R$ containing at most $r^2$ 
vectors  
such that $\rho$ can be written 
$$
\rho=\sum_{k=1}^{r^2}p_k^\prime\cdot\omega_k^\prime
$$
where the $p_k^\prime$ are nonnegative numbers with sum $1$ and the $\omega_k^\prime$ are 
pure product states associated with vectors in $S$.  Assuming now 
that $\rho\in E^r(\omega)$, then $\rho$ restricts to a {\em faithful} 
state of $\mathcal B(H)$ and hence $r^2\geq n$.  It follows that {\em $E^r(\omega)$ contains no 
separable states when $r<\sqrt n$.}
I am indebted to an anonymous 
referee for pointing out the idea behind this observation.  
\end{rem}

\section{Entangled states of large rank}\label{S:di}

Let $H$, $K$ be Hilbert spaces with $n=\dim H\leq m=\dim K<\infty$.  We conclude 
with an observation showing that the behavior of Theorem \ref{erThm1} 
does not persist through states of large rank.  While the first 
sentence of Theorem \ref{diThm1} is essentially known (for example, see 
\cite{gurvBarBall},  \cite{gurvBar}), 
we sketch a proof for completeness.

\begin{thm}\label{diThm1}
The set of separable states of $\mathcal B(K\otimes H)$ of rank $mn$ contains 
a nonempty relatively open subset of the state space of $\mathcal B(K\otimes H)$.  

Moreover, for every faithful state 
$\omega$ of $\mathcal B(H)$, the set of entangled 
states of $E^{mn}(\omega)$ is a relatively open subset that is 
neither empty nor dense in $E^{mn}(\omega)$, and its 
probability $p$ satisfies $0<p<1$.    
\end{thm}

\begin{proof}  Note first that the set of {\em faithful} separable states 
must linearly span the self adjoint part
$S$  of 
the dual of $\mathcal B(K\otimes H)$; equivalently, for every nonzero self adjoint operator 
$x$, there is a faithful separable state $\omega$ such that $\omega(x)\neq 0$.  
Indeed, fixing $x$, we use the fact that the separable states obviously span $S$ 
to find a separable state 
$\omega$ for which $\omega(x)\neq 0$, and then we can make small changes in 
the decomposable vector states that sum to $\omega$ 
so as to find a faithful separable state $\omega^\prime$ 
close enough to $\omega$ that $\omega^\prime(x)\neq 0$.  
Since the separable states of rank $mn$ span $S$, 
we can find a basis for $S$ consisting of separable states of rank $mn$.  

Finally, since the 
convex hull of a basis for $S$ consisting of states 
must contain a nontrivial open subset of the state 
space of $\mathcal B(K\otimes H)$, 
it follows that $\Sep(X^{mn})$ has nonempty 
interior and therefore has positive $P^{mn}$-measure.   Theorem 
\ref{psThm1} implies $0<P^{mn}(\Sep(X^{mn}))<1$, and the remaining assertions of 
Theorem \ref{diThm1} follow.  
\end{proof}

\section{Constructibility, Entanglement, and Zero-One laws}\label{S:ec}

In this section we digress in order 
to make some observations about set-theoretic issues that seem to add perspective to the results 
of Sections \ref{S:er} and \ref{S:di}, and which address the broader question of whether 
entanglement can be detected by way of a more detailed analysis of real-analytic varieties.  

Let $H$, $K$ be Hilbert spaces with $n=\dim H\leq m=\dim K<\infty$ and 
fix $r=1,2,\dots, mn$.  The subvarieties of $V^r(H,K)$ (see  Definition \ref{nsDef1})
generate a $\sigma$-algebra $\mathcal A$ of subsets of $V^r(H,K)$.  This $\sigma$-algebra 
consists of Borel sets and it separates points of $V^r(H,K)$.    In the context of descriptive set theory, 
$\mathcal A$ consists of all Borel sets that can be constructed by way of a 
transfinite hierarchy 
of operations consisting of countable unions and complementations, 
starting with subvarieties.  Let $\mathcal B$ be the somewhat larger $\sigma$-algebra consisting 
of all Borel sets $E\subseteq V^r(H,K)$ which agree almost surely with sets of $\mathcal A$ in that 
there are sets $A_1,A_2\in \mathcal A$ such 
that $A_1\subseteq E\subseteq A_2$ and $\mu(A_2\setminus A_1)=0$, $\mu$ being the natural probability 
measure on $V^r(H,K)$.  

Significantly, the ``constructible" sets in $\mathcal A$ and $\mathcal B$ 
satisfy a zero-one law.  

\begin{prop}\label{ecProp1}
For every $E\in\mathcal B$, $\mu(E)=0$ or $1$.  
\end{prop}

\begin{proof} It clearly suffices to show that $\mu\restriction_\mathcal A$ is 
$\{0,1\}$-valued.  To prove that, let $\mathcal Z$ be the family of all {\em proper} 
subvarieties $Z\neq V^r(H,K)$.  By Proposition \ref{nsProp1}, every set in $\mathcal Z$ has 
measure zero.  Consider the family $\mathcal C$ of all Borel subsets $E\subseteq V^r(H,K)$ 
with the property that either $E$ or its complement is contained in some countable union  
$Z_1\cup Z_2\cup\cdots$ of sets $Z_k\in \mathcal Z$.  One checks easily that $\mathcal C$ is closed 
under countable unions, complementation, and it contains $\mathcal Z$.  Hence $\mathcal C$ 
is a $\sigma$-algebra containing $\mathcal A$.  
But for every set $E\in \mathcal C$ we have $\mu(E)=0$ if $E$ is contained in a countable union 
of sets from $\mathcal Z$, or $\mu(E)=1$ if the complement of $E$ is contained in a countable 
union of sets from $\mathcal Z$.  Hence $\mu(E)=0$ or $1$.  In particular,  $\mu\restriction_\mathcal A$ 
is $\{0,1\}$-valued.  
\end{proof}

Now fix  
a faithful state $\omega$ of $\mathcal B(H)$, fix $r=1,2,\dots,mn$, 
and consider the space of all separable states 
in $E^{r}(\omega)$. The inverse image of this space 
under the parameterizing map $v\in V^r(H,K)\mapsto \rho_v\in E^r(\omega)$, namely  
$$
\Sep(V^r(H,K))=\{v\in V^r(H,K): \rho_v {\rm{\ is\ separable }}\}, 
$$ 
is a compact subspace of 
$V^r(H,K)$.  Proposition \ref{psProp3} shows that its 
structure determines the properties of separable states in $E^r(\omega)$, 
and its complement determines the properties of entangled states in $E^r(\omega)$.  

\begin{rem}[Structure of $\Sep(V^r(H,K))$ for small $r$]
The key fact in the proof of Theorem \ref{erThm1} is that for relatively small values 
of $r$, $\Sep(V^r(H,K))$ is contained in a proper subvariety $A^*$.  
It follows that {\em $\Sep(V^r(H,K))$ belongs to the $\sigma$-algebra $\mathcal B$ 
when $r$ satisfies $1\leq r\leq n/2$.  }
\end{rem}

\begin{rem}[Structure of $\Sep(V^r(H,K))$ for large $r$]
On the other hand, for large values of $r$ the set $\Sep(V^r(H,K))$ 
has different properties.    
Indeed, Theorem \ref{diThm1} asserts that the probability of $\Sep(V^{mn}(H,K))$ is 
neither $0$ nor $1$, so that  Proposition \ref{ecProp1} implies that $\Sep(V^{mn}(H,K))$ cannot belong to the 
$\sigma$-algebra $\mathcal A$ of 
``real-analytically constructible" sets, 
nor even to its somewhat larger relative $\mathcal B$.  Perhaps this set-theoretic phenomenon helps to 
explain the computational difficulties that arise from attempts to 
decide whether a concretely presented state of a tensor product of matrix algebras is entangled.  
\end{rem}

Finally, note that for any $r$,  (\ref{wiEq1}) implies that 
$\Sep(V^r(H,K))$ can be expressed 
as an uncountable union of proper subvarieties $\cup\{Z_\lambda: \lambda\in U(q)\}$ 
parametrized by the group $U(q)$, $q=m^2n^2$. 
But since the union is uncountable, 
that fact provides no information about whether $\Sep(V^r(H,K))$ belongs 
to the constructible $\sigma$-algebra 
$\mathcal A$.

\section{Concluding remarks}\label{S:cr}

\begin{rem}[States versus completely positive maps]\label{crRem1}
While we have focused on states of matrix algebras and their 
extensions in this paper, all of the above results have 
equivalent formulations as statements about completely positive maps.  In more concrete terms, 
note that with every $r$-tuple $v=(v_1,\dots,v_r)\in V^r(H,K)$ one can associate a 
unit-preserving completely positive (UCP) map $\phi_v: \mathcal B(K)\to \mathcal B(H)$ by way of 
$$
\phi_v(a)=\sum_{k=1}^r v_k^*av_k,\qquad a\in\mathcal B(K),   
$$
and there is a simple notion of {\em rank} in the category of completely positive maps in 
which $\phi_v$ has rank $\leq r$ (see \cite{arvMono}, Remark 9.1.3).  Indeed, this map 
promotes to a homeomorphism $\dot v\in X^r\mapsto \phi_v$ of $X^r$ onto the space 
of UCP maps of rank $\leq r$.   This parameterization $v\mapsto \phi_v$ of UCP maps of rank $\leq r$ 
corresponds to the parameterization $v\mapsto \rho_v\in E^r(\omega)$ of (\ref{ssEq2}) via    
\begin{equation}\label{crEq1}
\rho_v(a\otimes b)=\langle (\phi_v(a)\otimes b)\xi_\omega,\xi_\omega\rangle, 
\qquad a\in\mathcal B(K), \quad b\in\mathcal B(H).  
\end{equation}
Indeed, the bijective correspondence (\ref{crEq1}) between states and UCP maps exists independently of the issues 
taken up in this paper, 
and it is useful.  

For example, the connection  
between states of $A\otimes M_n$ (where $A$ is a unital \cstar) 
and completely positive maps 
of $A$ into $M_n$ was first exploited 
in the proof of the extension theorem for completely positive 
maps (see Lemma 1.2.6 of \cite{arvSubalgI}).  Shortly after 
\cite{arvSubalgI} appeared, this connection 
was made more explicit and 
further exploited by the author and George Elliott (independently, and in 
both cases unpublished), 
so as to reduce the 
extension theorem for operator valued completely positive maps 
(Theorem 1.2.3 of \cite{arvSubalgI}) to Krein's extension 
theorem for positive linear functionals.  In the intervening 40 years, the connection 
has been rediscovered more than once, and has 
found its way into the lore of completely positive maps and quantum information 
theory (see \cite{rudQs} and references therein).   
\end{rem}

\begin{rem}[Quantum channels]
A {\em quantum channel} is a completely positive map $\psi:M^\prime\to N^\prime$ 
between the {\em duals} of matrix algebras $M$ and $N$ that carries states to states.  
Quantum channels are the adjoints of UCP maps.  Indeed,  the most general 
quantum channel $\psi$ as above has the form $\psi(\rho)=\rho\circ\phi$, $\rho\in M^\prime$, 
where $\phi: N\to M$ is a UCP map.  In 
particular, {\em quantum channels of rank $\leq r$ 
are parameterized by the same real-analytic 
noncommutative sphere that serves to parameterize UCP maps of rank $\leq r$.  }
\end{rem}

\begin{rem}[Better estimates of the critical rank]\label{crRem2}
Fix Hilbert spaces $H$, $K$ of dimensions $n\leq m$ respectively, and 
let $\nu(n,m)$ be the largest integer  
such that the probability of entanglement in $(X^r,P^r)$ is 
$1$ for every $r=1, 2, \dots,\nu(n,m)$.  Together, Theorems \ref{erThm1} and \ref{diThm1} make the 
assertion 
$$
n/2\leq \nu(n,m)<nm.  
$$  
Our feeling is that each of these two bounds is far from best possible, and 
the problem of improving these bounds deserves further study.  
\end{rem}

\begin{rem}[Bitraces]
By a {\em bitrace} we mean a state $\rho$ of $\mathcal B(H\otimes H)$ such that 
$\rho(a\otimes\mathbf 1)=\rho(\mathbf 1\otimes a)=\tau(a)$, $a\in\mathcal B(H)$, 
$\tau$ being the tracial state of $\mathcal B(H)$.   
There has been recent work on identifying the extremal bitraces, 
of which we mention only \cite{parthaEx}, \cite{priSak} and, in the equivalent context of UCP maps, 
\cite{LandStr}.  
 After associating bitraces with UCP maps as in (\ref{crEq1}), one finds 
that bitraces are in one-to-one correspondence with the set of all UCP maps $\phi: \mathcal B(H)\to \mathcal B(H)$ 
that preserve the trace.  In turn, the space of 
all trace-preserving UCP maps of rank $\leq r$ corresponds to 
the subspace of $V^r(H,H)$ consisting of all $r$-tuples $v=(v_1,\dots,v_r)$ that satisfy 
$$
v_1^*v_1+\cdots+v_r^*v_r=v_1v_1^*+\cdots+v_rv_r^*=\mathbf 1_H.  
$$
The latter equations define a proper subvariety of $V^r(H,H)$ (Definition \ref{nsDef1}) 
that is neither homogeneous nor 
connected, and whose structure is considerably more complicated that that of $V^r(H,H)$ itself.  It is 
unclear to what extent the results of this paper have counterparts for bitraces.  
\end{rem}

\appendix
\section{Existence of real-analytic structures}\label{S:a1}

Theorem \ref{a1Thm2} below is essentially known; but since it is basic to our 
main result, we include a proof.  The argument we give 
makes use of the following result, which paraphrases a special case of 
Theorem 10.3.1 of \cite{dieuAnal}.  It asserts that  a real analytic map 
of $\mathbb R^n$ to $\mathbb R^m$ 
whose derivative has constant rank can be realized locally as a linear map 
$L:\mathbb R^n\to \mathbb R^m$ after a 
real-analytic distortion of both coordinate systems.  Let $U,V$ be open subsets 
of $\mathbb R^n$.  A real-analytic isomorphism of $U$ on $V$ is a bijection 
$u: U\to V$ such that 
both $u$ and $u^{-1}$ are real-analytic mappings.

\begin{thm}\label{a1Thm1}
Let $D\subseteq \mathbb R^n$ be an open set and let $f: D\to \mathbb R^m$ be 
a real-analytic mapping such that $\rank f^\prime(x)=r$ is constant for 
$x\in D$.  Then for every $a\in D$, there exist 
\begin{enumerate}
\item[(i)]
a real-analytic isomorphism 
$u$ of the open unit ball of $\mathbb R^n$ onto an open set $U\subseteq \mathbb R^n$ satisfying 
$a\in U\subseteq D$,  
\item[(ii)]
a real-analytic isomorphism $v$ of the open unit 
ball of $\mathbb R^m$ onto an open set $V\subseteq \mathbb R^m$ satisfying 
$f(U)\subseteq V$,  
\end{enumerate}
such that $f\restriction_U$ admits a factorization $f=v\circ L\circ u^{-1}$, where $L:\mathbb R^n\to\mathbb R^m$ 
is the linear map $L(x_1,\dots,x_n)=(x_1,\dots,x_r,0, \cdots,0)$.  
\end{thm}

\begin{thm}\label{a1Thm2}
Let $H$, $K$ be finite-dimensional Hilbert spaces with $\dim H\leq \dim K$.  Then the 
space $\mathcal S$ of all isometries in $\mathcal B(H,K)$ is a connected real-analytic 
manifold, and a homogeneous space relative to a smooth transitive action of the unitary 
group $\mathcal U(K)$.  In particular, there is a unique probability measure on $\mathcal S$ 
that is invariant 
under the $\mathcal U(K)$-action.  
\end{thm}

\begin{proof}  
To introduce a real-analytic structure on $\mathcal S$, 
consider the mapping $f: \mathcal B(H,K)\to \mathcal B(H)$ given 
by $f(v)=v^*v$.  If we view $f$ as a real-analytic map of finite-dimensional real vector 
spaces, then the derivative of $f$ at $v\in\mathcal B(H,K)$ is the real-linear map 
$f^\prime(v): h\in\mathcal B(H,K)\mapsto v^*h+h^*v\in\mathcal B(H)$.  The range of $f^\prime(v)$ is 
contained in the real vector space $\mathcal B(H)^{\text sa}$ of self-adjoint operators on $H$.  

Let $D$ be the set of all 
$v\in\mathcal B(H,K)$ such that $v^*v$ is invertible.  Then $D$ is an open set 
containing $\mathcal S$, and we claim that $f^\prime(v)$ has range $\mathcal B(H)^{\text sa}$ 
for every $v\in D$.  Indeed, the most general real linear 
functional on $\mathcal B(H)^{\text sa}$ has the 
form $\omega(y)=\tr(\Omega y)$ for some $\Omega=\Omega^*\in\mathcal B(H)$, and we have 
to show that if $\omega$ annihilates the range of $f^\prime(v)$ for some $v\in D$ then $\omega=0$.  
Since $\Omega=\Omega^*$, 
we can replace $h$ with $\sqrt{-1}h$ in the formula 
$$
\tr(\Omega(v^*h+h^*v))=\omega(f^\prime(v)(h))=0
$$ 
to obtain $\tr(\Omega v^*h-h^*v))=0$.  After adding these two expressions we obtain 
$\tr(\Omega v^*h)=0$ 
for 
all $h\in\mathcal B(H,K)$, hence $\Omega v^*=0$ for all $v\in D$.   
It follows that $\Omega v^*v=0$ and finally $\Omega=0$ since $v^*v$ is invertible for every $v\in D$.  

Hence the rank of $f^\prime(v)$ is constant throughout $D$. Theorem \ref{a1Thm1} now implies 
that the subspace $\mathcal S=\{v\in D: f(v)=\mathbf 1_H\}$ 
of $D$ can be endowed locally with a real-analytic 
structure, and moreover, that these local structures are mutually compatible with each other.  
Hence $\mathcal S$ is a real-analytic submanifold of $\mathcal B(H,K)$.  

For the remaining statements, fix  
$u,v\in\mathcal S$.  We claim that there is a unitary operator $w\in\mathcal B(K)$ 
such that $wu=v$.  Indeed, 
since $\|u\xi\|=\|v\xi\|=\|\xi\|$ for every $\xi \in H$, 
we can define an isometry $w_0$ from the range of $u$ to the range of $v$ 
by setting 
$w_0(u\xi)=v\xi$ for all $\xi\in H$.  Since $K$ is finite-dimensional, 
$w_0$ can be extended to a unitary operator $w\in\mathcal U(K)$, and 
$w$ satisfies $wu=v$.  It follows that the natural action of $\mathcal U(K)$ 
on $\mathcal S$ is smooth and transitive.  

The preceding observation implies that $\mathcal S$ is arcwise connected.  Indeed, for any 
two isometries $u,v\in \mathcal S$, there is a unitary operator $w\in\mathcal U(K)$ 
such that $wu=v$; and since the unitary group of $K$ is arcwise connected, it follows 
that $u$ can be connected to $v$ by an arc of isometries.  
\end{proof}

\begin{rem}[Identification of the invariant measure on $\mathcal S$]\label{a1Rem1}
The $\mathcal U(K)$-invariant probability measure $\mu$ on $\mathcal S$ can be described more 
concretely as follows.  The space $\mathcal S$ is embedded in the space of all operators 
$\mathcal B(H,K)$, and we can view the latter as a real Hilbert space with inner product 
$$
\langle a,b\rangle = \Re \tr(b^*a),\qquad a,b\in \mathcal B(H,K).  
$$
The unitary group $\mathcal U(K)$ acts as isometries of this real Hilbert space 
by left multiplication  
$(u,a)\in \mathcal U(K)\times\mathcal B(H,K)\mapsto ua\in \mathcal B(H,K)$.  
In turn, since the tangent spaces of $\mathcal S$ are naturally embedded in $\mathcal B(H,K)$, 
this inner product gives rise to a Riemannian metric on $\mathcal S$, which in turn gives rise 
to a natural probability measure $\tilde \mu$ after renormalization.  Since the 
group $\mathcal U(K)$ acts as isometries relative to the Riemannian structure of 
$\mathcal S$, the measure $\tilde\mu$ must be invariant under the action of $\mathcal U(K)$, 
and hence  $\mu=\tilde\mu$. In particular, {\em $\mu$ is mutually absolutely continuous with 
Lebesgue measure in smooth local coordinate systems for $\mathcal S$.  }
\end{rem}

\section{Zeros of real-analytic functions}\label{S:a2}

While the result of this appendix is well known, we lack a convenient reference and 
include a simple proof, the idea of which shown to me by Michael Christ.

\begin{prop}\label{a2Prop1}
Let $D\subseteq\mathbb R^n$ be a connected open set and let $f:D\to \mathbb R$ 
be a real-analytic function that does not vanish identically.  Then the set of 
zeros of $f$ has Lebesgue measure zero.  
\end{prop}

\begin{proof}
Let $Z=\{x\in D: f(x)=0\}$.  It suffices to show that for every point 
$a\in D$ there is an open set $U$ containing $a$ such that $Z\cap U$ has measure zero.  
Choose a point $a\in D$.  The power series expansion of $f$ about $a$ cannot have 
all zero coefficients, since that would imply that $f$ vanishes on an open set, 
hence identically.  Therefore some mixed partial of $f$ of order $N$ must be nonzero 
at $a$.  This implies that the $N$th derivative of $f$ in some direction must be 
nonzero at $a$.  By rotating the coordinate system of $\mathbb R^n$ about $a$, we can assume that 
$\partial ^Nf/\partial x_1^N$ is nonzero at $a$, and therefore on some open rectangle 
$U$ centered at $a$.  Let $L$ be any line of the form $x_2=c_2,\dots,x_d=c_d$ where 
$c_2,\dots,c_d$ are constants.  If $L\cap U\neq \emptyset$, then the 
restriction of $f$ to $L\cap U$ is a nonzero real-analytic function 
of the single variable $x_1$ - which has isolated zeros.  Hence the intersection of $Z$ with $L\cap U$ 
has linear Lebesgue measure zero.  By Fubini's theorem, $Z\cap U$ has measure zero.  
\end{proof}

\subsection*{Acknowledgements} I want to thank David Gale and Mike Christ for 
helpful conversations concerning aspects of this paper.  Thanks to Mary Beth Ruskai 
for providing help with references and advice on other issues.  
I also thank an anonymous 
referee for suggesting a significant shortening of the original proof of Theorem \ref{diThm1} 
as well as for other useful comments.

\bibliographystyle{alpha}

\newcommand{\etalchar}[1]{$^{#1}$}
\newcommand{\noopsort}[1]{} \newcommand{\printfirst}[2]{#1}
  \newcommand{\singleletter}[1]{#1} \newcommand{\switchargs}[2]{#2#1}

\end{document}